\documentclass[10pt,a4paper]{article}
\usepackage[latin1]{inputenc}
\usepackage[T1]{fontenc}
\usepackage{amsmath, amsthm}
\usepackage{amsfonts}
\usepackage{amssymb}
\usepackage{bbold}
\usepackage{mathtools}
\usepackage{graphicx}
\usepackage{color}
\usepackage[margin=1.25in]{geometry}
\usepackage{enumerate}
\usepackage[symbol,perpage]{footmisc}
\usepackage{authblk}

\newtheorem{theorem}{Theorem}
\newtheorem{corollary}[theorem]{Corollary}
\newtheorem{proposition}[theorem]{Proposition}
\newtheorem{lemma}[theorem]{Lemma}
\newtheorem{conjecture}[theorem]{Conjecture}

\theoremstyle{definition}
\newtheorem{definition}{Definition}[section]

\numberwithin{theorem}{section}

\newcommand{\dd}{\mathrm{d}}
\newcommand{\cha}{\mathbb{1}}

\newcommand{\R}{\mathbb{R}}
\newcommand{\N}{\mathbb{N}}
\newcommand{\Z}{\mathbb{Z}}
\DeclareMathOperator{\sign}{sign}

\DeclareMathOperator{\id}{id}
\DeclareMathOperator{\Leb}{Leb}

\DeclareMathOperator{\supp}{supp}
\newcommand{\Wul}{\mathcal{W}^u_{\rm loc}}
\newcommand{\C}{\mathcal{C}}

\newcommand{\M}{\mathcal{M}}
\renewcommand{\S}{\mathcal{S}}
\newcommand{\condmeas}{\rho_\S}

\renewcommand{\L}{\mathcal{L}}

\renewcommand{\O}{\mathcal{O}}

\newcommand{\Aord}{1} 

\author{Caroline L. Wormell\thanks{Laboratoire de Probabilit\'es, Statistique et Mod\'elisation,
		Sorbonne Universit\'e, CNRS \\
		email: {\sf wormell@lpsm.paris}, {\sf ca.wormell@gmail.com}\\
		ORCID: 0000-0003-2953-6493}
}
\title{Conditional mixing in deterministic chaos} 
\begin{document}
	\maketitle
	
	\begin{abstract}
		While on the one hand, chaotic dynamical systems can be predicted for all time given exact knowledge of an initial state, they are also in many cases rapidly mixing, meaning that smooth probabilistic information (quantified by measures) on the system's state has negligible value for predicting the long-term future. However, an understanding of the long-term predictive value of intermediate kinds of probabilistic information is necessary in various physical problems, and largely remains lacking. 
		
		Of particular interest in data assimilation and linear response theory are the conditional measures of the SRB measure on zero sets of general smooth functions of the phase space. In this paper we give rigorous and numerical evidence that such measures generically converge back under the dynamics to the full SRB measures, exponentially quickly. We call this property conditional mixing. We will prove that conditional mixing holds in a class of generalised baker's maps, and demonstrate it numerically in some non-Markovian piecewise hyperbolic maps. Conditional mixing provides a natural limit on the effectiveness of long-term forecasting of chaotic systems via partial observations, and appears key to proving the existence of linear response outside the setting of smooth uniform hyperbolicity.
	\end{abstract}
	
	
	By definition, chaotic dynamical systems' future states are generally impossible to predict over the long term. For this reason they must be studied probabilistically, that is, in terms of measures that evolve under the dynamics.
	Many probabilistic questions about topologically mixing chaotic systems $f: M \circlearrowleft$ are quantitatively related to fast mixing (i.e. decay of correlations) for smooth observables with respect to some invariant measure $\mu$ \cite{Baladi00}. Mathematically, mixing means the following decay of sufficiently regular (e.g. $C^\infty$) observables $A, B: M \to \R$:
	\[ \int_M A \circ f^n \, B\, \dd\mu - \int_M A\,\dd\mu \int_M B\,\dd\mu \xrightarrow{n\to\infty} 0. \]
	Alternatively, this can be reformulated as weak convergence of smoothly re-weighted versions of $\mu$ back to $\mu$ under the actions the dynamics:
	\[ \int_M A \, \dd f^n_* (B \mu) \xrightarrow{n\to\infty} \int_M A\,\dd\mu \]
	for any $A,B$ with $\int_M B\,\dd\mu = 1$, where $f_*$ pushes measures forward under $f$. 
	
	A particularly important invariant measure is the SRB measure $\rho$, to which, additionally, measures with smooth enough Lebesgue densities are expected to converge (typically exponentially quickly), thus establishing it as the physically relevant invariant measure:
	\begin{equation} \int_M A \, \dd f^n_* (B \mu) \xrightarrow{n\to\infty} \int_M A\,\dd\rho \label{eq:LebesgueConvergence} \end{equation}
	for all $A,B \in C^\infty$ with $\int_M B\,\dd\mu = 1$, where $\mu$ is Lebesgue measure. It is in fact standard that for sufficiently hyperbolic maps with $d_u$ positive Lyapunov exponents, \eqref{eq:LebesgueConvergence} also holds for any $\mu$ with a regular conditional density along $d_u$-dimensional submanifolds that are tangent to expanding direction in phase space.
	
	
	However, establishing fast mixing (or the related decay of iterates of a transfer operator) appears to be insufficient to answer various questions that depend on the long-term behaviour of a small problematic subset of the system's attractor, such as the response problem in non-uniformly hyperbolic systems \cite{Ruelle18}. We might therefore ask if other measures $\mu$ have the same property as Lebesgue measure in \eqref{eq:LebesgueConvergence}.
	
	Perhaps the simplest to describe example of a small subset of an attractor would be its restriction to a submanifold of phase space---which will generically be transverse to both stable and unstable manifolds. If this submanifold comes from some foliation (e.g. of level sets of an observable), we can disintegrate the physical measure and study it's conditional measure $\mu$ on this submanifold. An example of such a measure is shown in Figure~\ref{fig:CondMeasure}: note that unlike Gibbs invariant measures, the measure is supported on a non-invariant Cantor set without a product structure.
	
	If \eqref{eq:LebesgueConvergence} obtains for such a conditional measure $\mu$, we call it conditional mixing. For dissipative systems, conditional mixing appears to lie outside the scope of traditional study by transfer operators, and this basic problem has hitherto seen very little study, notwithstanding work for other classes of $\mu$ in the specific case of linear one-dimensional maps \cite{Hochman15,Bourgain17, Sahlsten20}
	
	\begin{figure}
		\centering
		\includegraphics{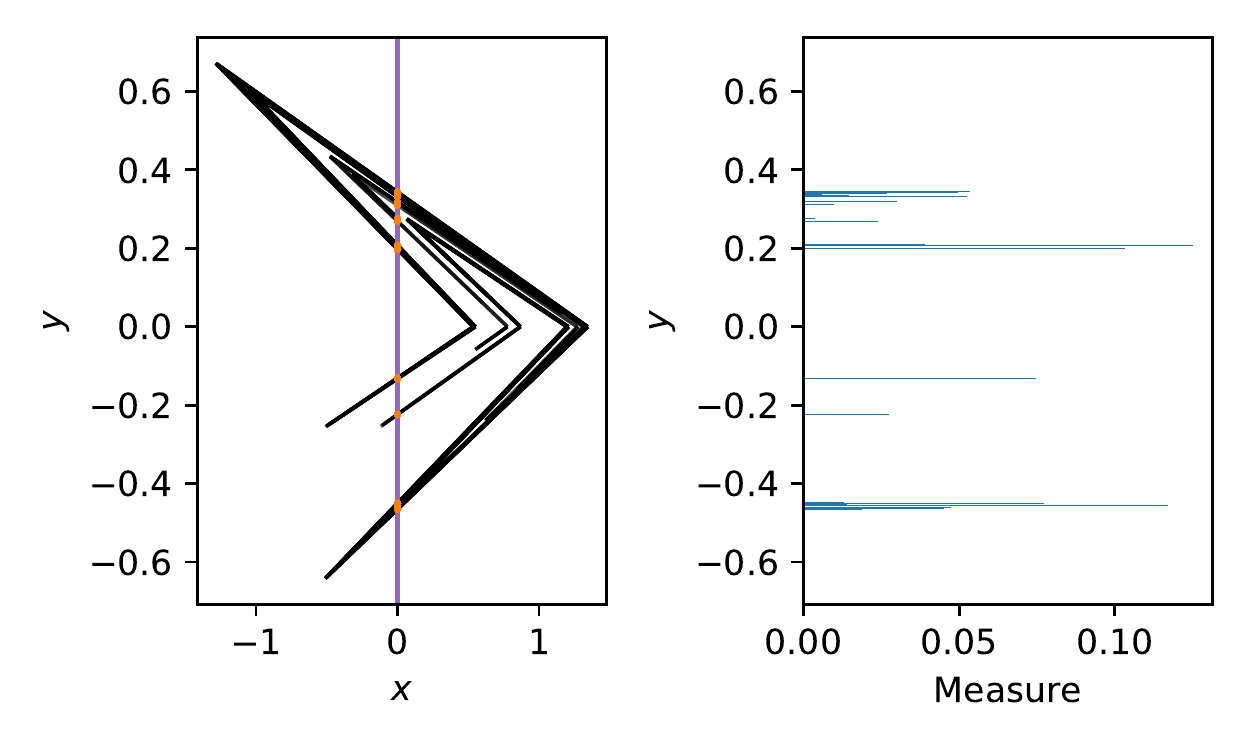}
		\caption{Left: picture of the Lozi attractor at $a=1.7, b=0.5$ (black), the singular line $\{x = 0\}$ (purple line), their intersection (orange Cantor set). Right: histogram of $\rho(y\mid 0)$, obtained from $200,\!000$ iterates of the unstable manifold dynamics $\vec{f}$, binned at width $0.0025$.}
		\label{fig:CondMeasure}
	\end{figure}
	
	Nevertheless, through a mixture of rigorous and theoretical study we will show that conditional mixing holds for a range of maps, and in particular, has some connection with the fractal-geometric theory of Fourier dimension \cite{Sahlsten20,Mosquera18}. From our results it seems that conditional mixing is likely to hold for a large set of maps and submanifolds (perhaps even almost all those for which there is no ``obvious'' reason why it could not hold). 
	
	We also describe some consequences of conditional mixing. One physically meaningful consequence of conditional mixing is that the capability of Bayesian filters to make predictions about chaotic systems in the long term is limited if they only make partial observations (see Section~\ref{s:PartialObs}); it also implies set-filling results for chaotic attractors (see Section~\ref{s:Definition}). Another physical application, to proving the widely-believed existence of linear response outside of smooth hyperbolic systems, will be considered in \cite{lozi}.
\\

	The paper is structured as follows: in Section~\ref{s:Definition} we give a mathematical definition of conditional mixing and give a simple, illustrative consequence of it involving set filling; in Section~\ref{s:PartialObs} we present the application of conditional mixing to the area of forecasting. In Sections~\ref{s:BakerResults} and~\ref{s:LoziResults} we give evidence for conditional mixing in various systems, respectively presenting a theorem for a class of (potentially nonlinear) baker's maps and numerical evidence for some piecewise hyperbolic maps. We discuss our results in Section~\ref{s:Conclusion}. The novel algorithms we use to obtain for our numerical results are given in Appendix~\ref{s:Algorithms}.
		
	\section{Definition and an illustrative consequence}\label{s:Definition}

	Let $T: M\to M$ be a dynamical system with SRB measure $\rho$, and let $H: M \to \R^d$ be a $C^2$ function with no critical points on the level set 
	\[\ell_H := \{x \in M : H(x) = 0\}.\] 
	Suppose that for some $c:\R^+ \to \R^+$ the following limit exists in the $C^0$-weak topology:
	\[ \mu = \lim_{\delta\to 0} c(\delta)^{-1} \mathbb{1}(|H(\cdot)|\leq \delta) \rho \]
	and that $\mu$ is a finite measure. Any two sequences $c, c'$ will yield $\mu, \mu'$ identical up to scaling, as would changing the kernel $k(\delta^{-1} |H(x)|) = \mathbb{1}(\delta^{-1}|H(x)| < 1)$ to another bounded, compactly-supported decreasing function. If $\mu(x)$ is a probability measure then we denote it by $\rho(x \mid H(x) = 0)$.
	
	\begin{definition}
		$(T,\rho)$ has {\it conditional mixing} with respect to $H$ if for all $A, B \in C^\infty$,
		\begin{equation} \left|\int_M A\circ T^n\,B\, \dd\mu - \int_M A\,\dd\rho \int_M B \dd\mu\right| \xrightarrow{n\to\infty} 0. \label{eq:CondMix} \end{equation}
		
	\end{definition}

	This is to say that repeatedly pushing forward any weighted conditional measure $B\mu$ by $T$, we converge back to the SRB measure $\rho$ (up to a multiplicative constant) in the weak topology with respect to $C^1$.

	It is natural to wish to put some quantitative bounds on mixing. By analogy with the usual sort of mixing, a natural decay rate is exponential:
	\begin{definition}
		$(T,\rho)$ has {\it exponential conditional mixing} with respect to $H$ if there exist $\xi<1$, $C$ and $r<\infty$ such that for all $A, B \in C^r$,
		\[ \left|\int_M A\circ T^n\,B\, \dd\mu - \int_M A\,\dd\rho \int_M B\, \dd\mu\right| \leq C \xi^n \|A\|_{C^r} \|B\|_{C^r}. \]
		
	\end{definition}

	Note that (exponential) conditional mixing can already be expected to hold from the classical theory if $f$ is conservative. For example, we have the following (proven in Appendix~\ref{a:CoveringProofs}):
	\begin{proposition}\label{p:ConservativeCM}
		Suppose $f$ is a $C^3$ conservative topologically mixing Anosov map. Then if $H \in C^2(M,\R^d)$ has no singular points on its zero set, the zero set is always transverse to stable manifolds, and $d$ is less than the number of stable directions, then $(f,\rho)$ has exponential conditional mixing with respect to $H$.
	\end{proposition}

	The measure-based conditional mixing implies an interesting set-convergence property of the intersection of the level set $\ell_H$ with the support of $\rho$, which we notate as $\Lambda$ and which is often an attractor of $T$. We find that iterates of the intersection of the line and $\Lambda$, i.e. iterates of a slice of $\Lambda$, converge back in Hausdorff distance to the full support of $\rho$. An example of this phenomenon is plotted in Figure~\ref{fig:Covering}.
	\begin{proposition}\label{p:HausdorffDistance}
		Conditional mixing with respect to $H$ implies that 
		\[ \lim_{n\to\infty} d_{\rm Haus}(T^n(\ell_H \cap \Lambda),\Lambda) = 0. \]
	\end{proposition}

	Under a reasonably general assumption on the regularity of the SRB measure, exponential conditional mixing gives us a quantitative version of this result as well:
	\begin{proposition}\label{p:ExponentialHausdorffDistance}
		Suppose $\rho$ is lower-Ahlfors regular: i.e. there exist $C,d$ such that $\rho(B(x,\delta)) > C \delta^d$ for all $x$. Exponential conditional mixing with respect to $H$ implies that for some $\xi_1 < 1$ and $C_1$,
		\[  d_{\rm Haus}(T^n(\ell_H \cap \Lambda),\Lambda) \leq C_1\xi_1^n. \]
	\end{proposition}
	The proofs of these two propositions are given in Appendix \ref{a:CoveringProofs}.

	\begin{figure}
		\centering
		\includegraphics{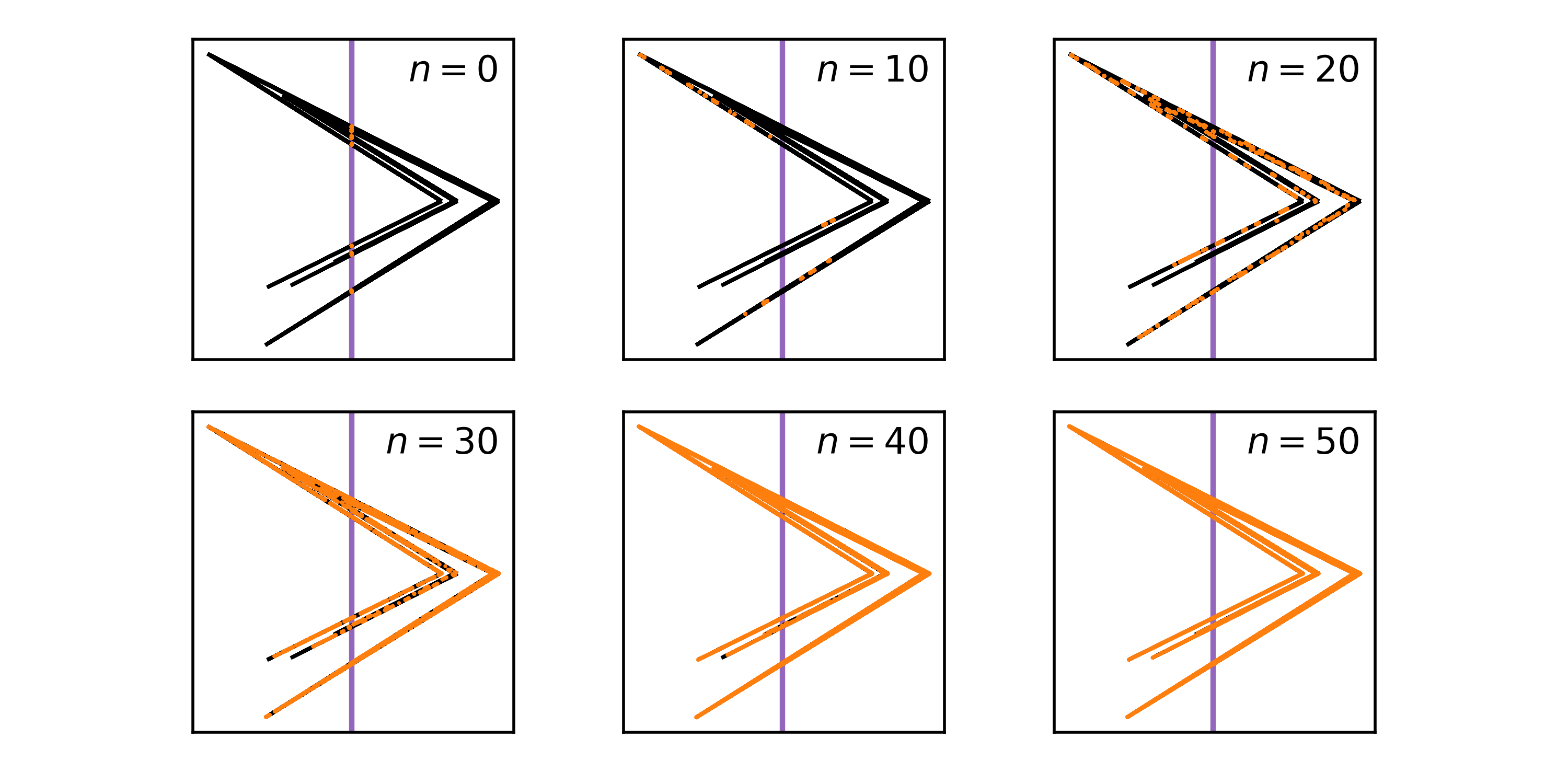}
		\caption{Exponentially fast filling of the Lozi attractor $\Lambda \subset \R^2$ (black) by the set intersection $\Lambda \cap \{x = 0\}$ (orange), pushed forward $n$ times by the same Lozi map. The level set $\{x = 0\}$ is given in purple. Lozi parameters $a = 1.8$, $b = 0.35$ are used.}
		\label{fig:Covering}
	\end{figure}

	\section{Forecasting with perfect partial observations}\label{s:PartialObs}

We now consider a fundamental practical problem to which the notion of conditional mixing is directly applicable: that of forecasting chaotic dynamics.
Many forecasting methods have been developed that assimilate information obtained from observations. In general, these methods achieve this assimilation by approximating the Bayesian filter, also known as the optimal filter \cite{Doucet01}. 

To understand this in the simplest instance, let us suppose that our system $T: \M \circlearrowleft$ for $\M \subset \R^d$ has exponential mixing, and at time $n=0$, we have some prior probabilistic knowledge of the state of our system, given by some (presumably ``nice'') measure $\dd\mu^{-}(x)$. If we start with an unobserved system at statistical equilibrium, the natural choice of prior is $\mu^{-} = \rho$, the SRB measure.

We can now make a noisy and perhaps partial observation of our system, given as a value $y = H(x) + \zeta \in \R^{e}$, where $\zeta$ is random with probabilities given by a smooth kernel $p(\zeta\mid x)\,\dd \zeta$. 

Assimilating the observation $y$, the posterior probability measure of $x$ is given by Bayes' theorem as
\begin{equation}
\dd\mu(x) = Z(y)^{-1} p(y - H(x)\mid x)\,\dd\mu^{-}(x), \label{eq:BayesTheorem}
\end{equation}
with normalising constant $Z(y) = \int_{M} p(y-H(w)\mid w)\,\dd\mu^{-}(w)$. An example is given in Figure~\ref{fig:BayesianAttractor}. Our best guess of the state of the system at future time $n$ (i.e. of $T^n(x)$) is then given by $T^n_*\mu$; the expected value of some (nice) observable $A$ at time $n$ is therefore
\begin{align} \int_M A(T^n(w))\,\dd\mu(w) = \int_M A(T^n(w))\, \frac{p(y-H(w)\mid w)}{Z(y)}\,\dd\mu^-(w).  \label{eq:ForwardPrediction} \end{align}

\begin{figure}[tb!]
	\centering
	\includegraphics{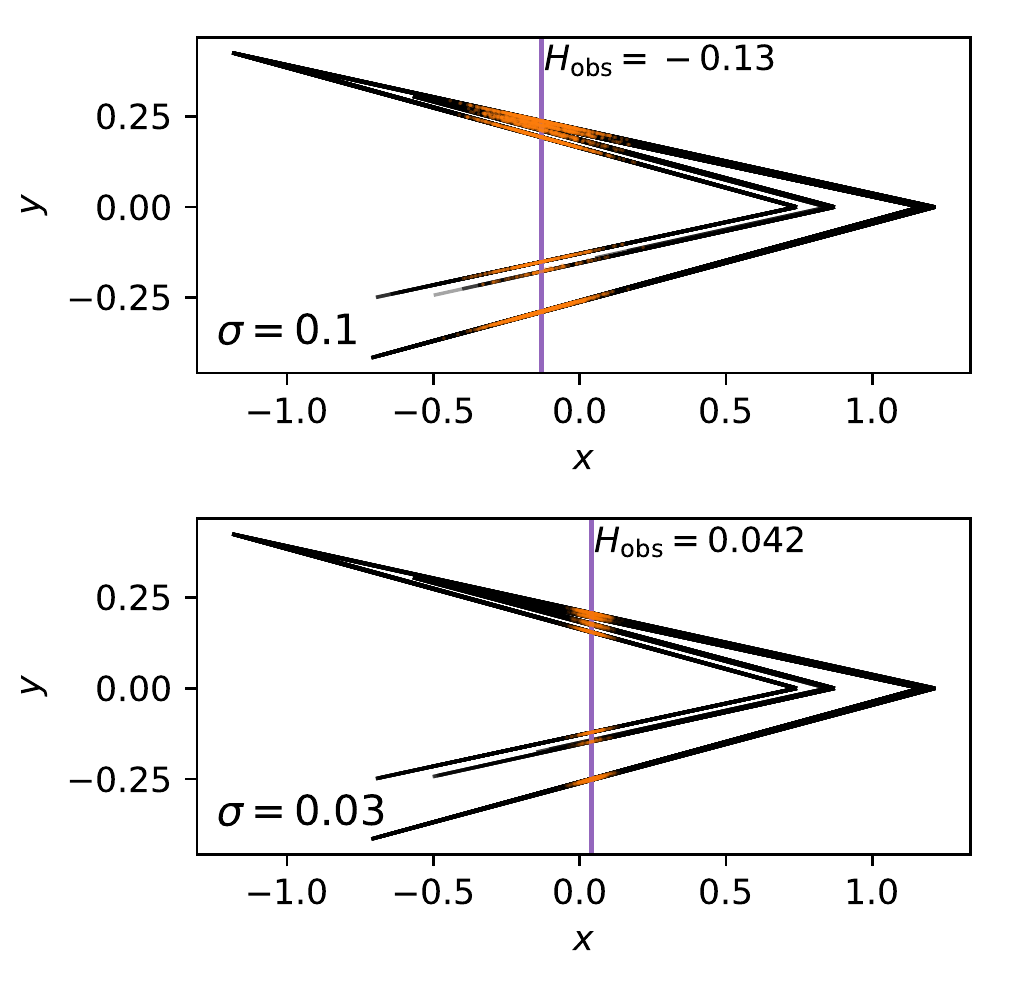}
	\caption{Posterior distribution (orange) after one observation step of $H(u,v) = u + \zeta$ where $\zeta \sim \mathcal{N}(0,\sigma^2)$, for varying values of $\sigma$. Line corresponding to $H(x,y) = H_{\mathrm{obs}}$ in purple, attractor in black. Convergence to the conditional measure as in Figure~\ref{fig:CondMeasure} can be seen as the noise $\sigma \to 0$.}
	\label{fig:BayesianAttractor}
\end{figure}

Depending on how effective a measurement $y=H(x)$ is of $x$, $\mu$ is likely to be concentrated on a smaller set than $\mu^-$, and this may improve forward estimates of the system's state over the short to medium term.
However, under our assumptions for $T, A, \mu^-, p$, exponential mixing results give that \eqref{eq:ForwardPrediction} will eventually converge at a fixed exponential rate to the SRB measure expectation $\int_M A\,\dd\rho$.

But as we make our observations more and more precise, i.e. reduce the noise in $H$ to zero, we might ask what posterior we end up with and what quality of forecast we can make with it. This is trivial if $H(x)$ specifies $x$: we will know our value of $x$ exactly, and therefore $T^n(x)$ exactly for all time. However, it is typical in high-dimensional systems for $H$ to be only a partial observation.

In a zero-noise limit the kernel $p$ is no longer smooth, with $p(\zeta\mid w) = \delta(\zeta)$. Our posterior measure $\dd\mu(x)$ is then the simply the conditional probability measure of $\dd\mu^{-}(x)$ given that $H(x) = y$:
\[\int_M A(f^n(w))\,\dd\rho(w\mid H(w) = y). \]

If conditional mixing \eqref{eq:CondMix} holds, this must converge to the expectation of $A$ with respect to the SRB measure $\rho$, i.e. in the long-term we end up back with the default no-information guess. Conditional mixing thus codifies the intuition that typical incomplete information on the system should wash out over time.

On the other hand, if conditional mixing did not hold generically, and thus partial observations could with some positive likelihood be predictively useful for all time, there might be significant practical consequences for chaotic systems. However, what our exploratory mathematical results in Sections~\ref{s:BakerResults} and~\ref{s:LoziResults} will suggest is that conditional mixing should in fact hold, excluding this possibility.
\\
In practice, of course, physical observations will always have some random error to them. This is also true of the evolution of physical systems, which are nonetheless considered worth studying in their zero noise limit. In any case, if an error is small enough it will take a while to manifest, and the zero-noise limit is what captures the medium term behaviour (see Figure~\ref{fig:Bayesian}). 

\begin{figure}
	\centering
	\includegraphics{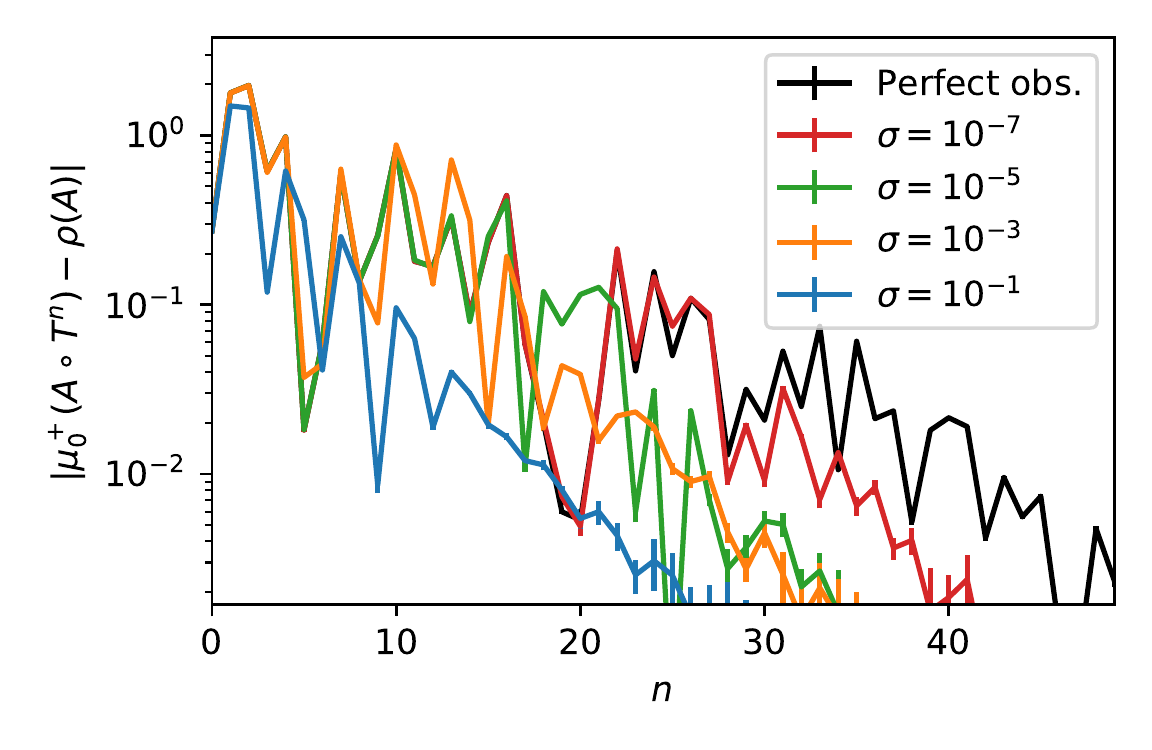}
	\caption{Decay of mean forecasts of $A(x,y) = 2x$ to SRB measure averages for Bayesian filters with various levels of observation noise. The observation function used was $H(x,y) = x$.}
	\label{fig:Bayesian}
\end{figure}

	\section{Rigorous results for a toy model: baker's map}\label{s:BakerResults}
	
As a simple model to study conditional mixing, let us consider baker's maps $b: D := [0,1]^2 \circlearrowleft$ of the following form:
	\begin{equation} \label{eq:BakerMap}
	b(x,y) = (kx \mod 1, v_{\lceil kx \rceil}(y))
	\end{equation}
	where $k\geq 2$ is an integer, the $v_i, i = 1,\ldots,k$ are (possibly nonlinear) contractions with all $|v_i'|\leq \mu < 1$. We will also assume that the open images $v_i((0,1))$ are disjoint, and the contractions have bounded distortion (i.e. the $\log |v_i'|$ are $C^1$). An example of such a map is plotted in Figure~\ref{fig:Baker}. For smooth enough transverse foliations, it is possible to define in a natural way conditional measures for all leaves:
	\begin{proposition}\label{p:ConditionalMeasures}
			Suppose that $\Psi(t;y) = (\psi(y) -t,y)$ is a foliation of some subset $D_\Psi \subseteq D$ for $|t| \leq t_* \in [0,1]$ and $y \in [0,1]$.
			Let $\rho$ be the (unique) SRB measure of $b$.
			
			 Then for every $t\in [-t_*,t_*]$ there exists a unique probability measure $\rho_t$ supported on $\Psi(t,[0,1])$ such that
			\begin{enumerate}[a.]
				\item For all $t \in (-t_*,t_*)$ and all continuous functions $A: D\to\R$,
				\begin{equation} \int_D A\,\dd\rho_t = \lim_{\delta\to 0} \frac{1}{2\delta} \int_{\left\{\Psi(s,y) :  |s-t|<\delta,y\in[0,1]\right\}} A\,\dd\rho, \label{eq:BakerSlice}\end{equation}
				\item For all Borel sets $E \subseteq D_\Psi$,
				\[ \rho(E) = \int_{-t_*}^{t_*} \rho_t(E)\,\dd t. \]
			\end{enumerate}
	\end{proposition}

	\begin{figure}
		\centering
		\includegraphics{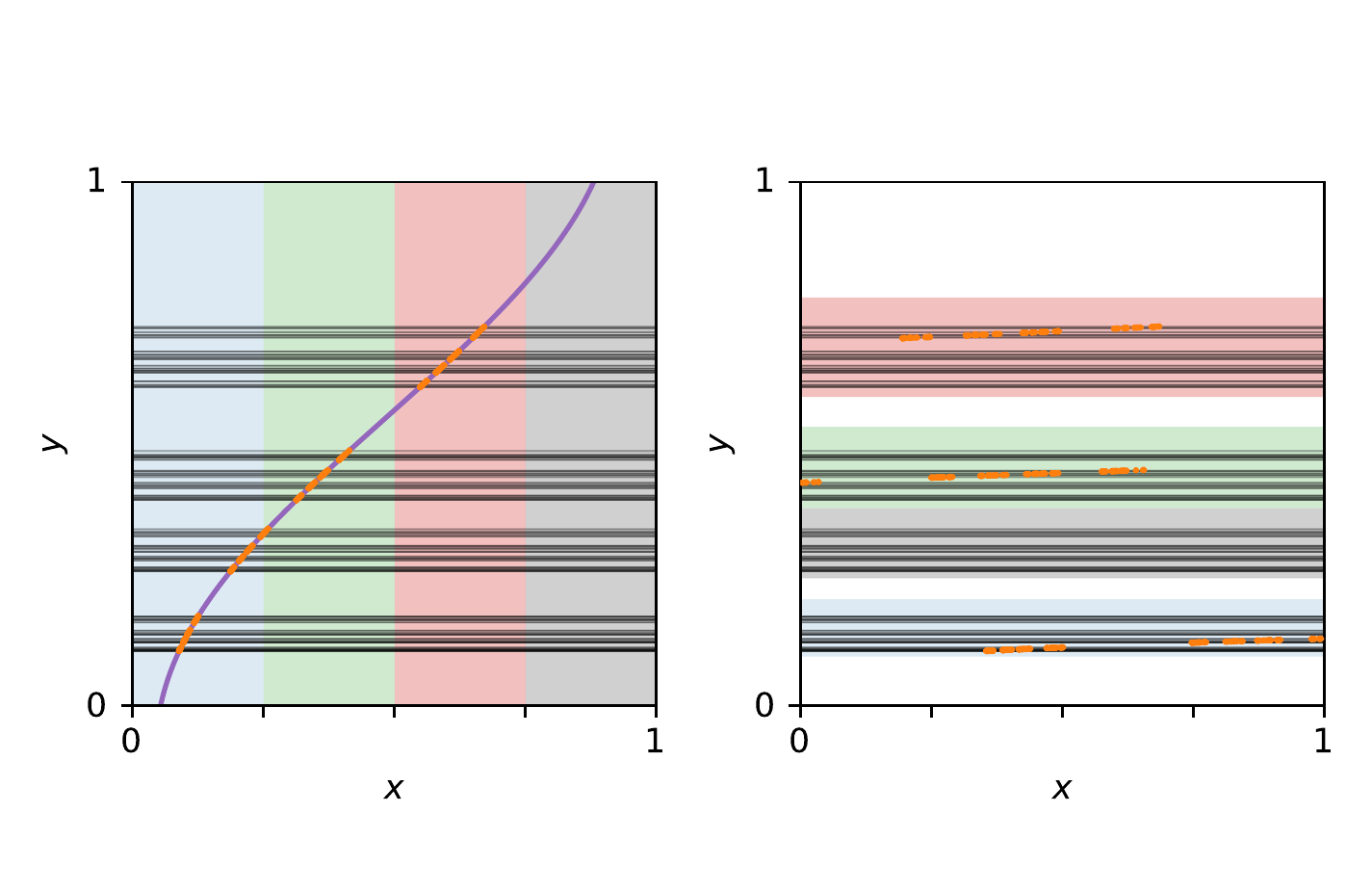}
		\caption{Picture of a baker's map of the form \eqref{eq:BakerMap} showing the attractor (black), an intersecting line (purple), and the action of the map on pieces of the domain (in pale colours), and on the conditional measure of the SRB measure on the intersecting line (orange).}
		\label{fig:Baker}
	\end{figure}
	
	Let us be as general as we can about the functions against which the conditional measures weakly converge back to the full measure. For $\alpha,\beta \in (0,1]$ we define the following norm on continuous functions $\phi: D \to \C$:
	\[ \| \phi \|_{\alpha;\beta} = |\phi|_{\alpha,x} + |\phi|_{\beta,y} + \|\phi\|_{L^\infty} \]
	where the directional H\"older semi-norms are given by
	\begin{align*}
	|\phi |_{\alpha,x} &= \sup_{x,x',y \in [0,1]} \frac{|\phi(x,y)-\phi(x',y)|}{|x-x'|^\alpha}\\
	|\phi |_{\beta,y} &= \sup_{x,y,y' \in [0,1]} \frac{|\phi(x,y)-\phi(x,y')|}{|y-y'|^\beta}.
	\end{align*}
	The Banach space $C^{\alpha;\beta}$ will then consist of all continuous functions $\phi: D \to \C$ with $\| \phi \|_{\alpha;\beta} < \infty$: that is, functions that are $\alpha$-H\"older in the $x$ direction and $\beta$-H\"older in the $y$ direction. In particular, the Banach space of $C^1$ functions is continuously embedded in $C^{\alpha;\beta}$.
	
	The following theorem, proved in Appendix~\ref{s:Baker}, gives that for certain baker's maps, exponential conditional mixing holds for all conditional SRB measures on a smooth foliation transversal to unstable lines:
	\begin{theorem}\label{t:Baker}
		Suppose that $\Psi(t;y) = (\psi(y) -t,y)$ is a foliation of some subset of $D$ for $|t| \leq t_* \in [0,1]$ and $y \in [0,1]$, and $\psi$ is $C^2$ with $\psi' \neq 0$. Suppose one of the following conditions hold:
		\begin{enumerate}[I.]
			\item The contractions $v_i$ are totally nonlinear\footnote{That is, no $C^2$ function exists conjugating all the $v_i$ simultaneously to linear functions.} and $C^2$ and $\cup_i v_i([0,1]) = [0,1]$. \label{opt:1}
			\item The contractions $v_i$ are totally nonlinear and analytic, as is $\psi$. \label{opt:2}
			\item The contractions are linear with $v_i(x) := \mu x + o_i$ for $o_i \in [0,1-\mu]$, and $\psi'' \neq 0$. \label{opt:3}
		\end{enumerate}
		
		Let $\rho$ be the (unique) SRB measure of a modified baker's map $b$ and let $\{\rho_t\}_{t\in[-t_*,t_*]}$ be the conditional measures of $\rho$ on the foliation. Then $(b,\rho)$ has exponential conditional mixing with respect to the level set of $H(x,y) = \psi(y)-t-x$ for all $t \in (-t_*,t_*)$.
		
		More specifically, there exists $d^* > 0$ such that for all $\gamma \in (1-d^*,1],\beta \in (0,1], \alpha \in (\gamma-d_*,1]$, there exist $C > 0$, $\xi \in (0,1)$ such that for all $t \in (-t_*,t_*)$, $A \in C^{\alpha;\beta}$, $B \in C^\gamma$, $n \in \N$,
		\[ | \rho_t(A\circ b^n\, B) - \rho_t(B) \rho(A) | \leq C \| A \|_{C^{\alpha;\beta}} \|B\|_{C^\gamma} \xi^n. \]
	\end{theorem}

	From this and Proposition~\ref{p:ExponentialHausdorffDistance}, exponential convergence in Hausdorff distance of the slice sets follow.
	\begin{corollary}\label{c:BakerCover}
		Under the conditions of Theorem~\ref{t:Baker}, there exist $C_1$ and $\xi_1$ such that for all $t \in (-t_*,t_*)$,
		\[ d_{\rm Haus}(b^n(\Psi(t;\R) \cap \Lambda_b),\Lambda_b) \leq C_1 \xi_1^n,\]
		where $\Lambda_b$ is the attractor of $b$.
	\end{corollary}

	Note that the classic piecewise affine baker's maps fall under condition~\ref{opt:3} of the theorem, although for conservative maps the application of Fourier dimension theory is unnecessary due to the Lebesgue-absolute continuity of $\rho$.

	It can also be seen from the proof that the constant $d_*$ is half the Fourier dimension of $\rho_t$ projected onto the $x$ coordinate. (A choice of test functions with more specific Fourier decay properties might yield $d_*$ to be exactly Fourier dimension). This Fourier dimension is bounded from above by the Hausdorff dimension of $\rho_t$, which, given the product structure of the measure, is easily seen also to be the stable dimension of the systems \cite{Schmeling98}. Thus, a larger stable dimension suggests conditional mixing holds with respect to increasingly less regular observables, with consequences for linear response theory \cite{Ruelle18, lozi}.
	
	In fact, for case~\ref{opt:3}, $d^*$ as well as the asymptotic rate of conditional mixing $\xi$ are independent of the conditioning foliation \cite[Theorem~3.1]{Mosquera18}.

	It is worth mentioning that it is also possible to decompose sufficiently regular (e.g. real analytic) curves that are tangent to stable or unstable manifolds away from the support of the attractor into a finite set of curves that uniformly avoid tangencies.
	\\
	
	To prove Theorem~\ref{t:Baker}, we use essentially two facts. The first fact, used in Lemma~\ref{l:Prop4}, is that the $x$ component of $b$ is a tupling map, whose action on Fourier coefficients of measures is well known. The second fact (Proposition~\ref{p:SRBMeasure}) is that the SRB measure is a product of uniform measure in the $x$ direction and a Gibbs measure (in fact, a measure of maximal entropy) of an iterated function system in the $y$ direction. This allows us to bring some recent results on Fourier dimension of Gibbs measures \cite{Mosquera18,Sahlsten20, Mosquera22}. 
	
	The Fourier dimension of Gibbs measures is an area in progress whose results have not yet been consolidated, hence the somewhat particular set of alternatives. We remark that if $B(x,y)$ depends only on $x$ then in case \ref{opt:3} the $\psi' \neq 0$ restriction can be dropped: that is, quadratic tangencies with stable manifolds (lines of constant $y$) are allowed here.

	While baker's maps' special structure (e.g. as skew products) allow us to find rigorous results for them by borrowing existing theory, it is not immediately clear how to mathematically generalise our results.

	\section{Numerical example: Lozi map}\label{s:LoziResults}
	
	To study maps with less structure we now therefore turn to rigorously justified numerics, and consider the commonly-studied and numerically amenable class of Lozi maps.
	
	These are piecewise hyperbolic affine maps $f: \R^2 \to \R^2$ with
	\begin{align}
	f(x,y) = (1 + y - a|x|, bx),\, b \neq 0. \label{eq:LoziDef}
	\end{align}
	
	For $a \in (1,2)$ and $b \in (0,\min\{a-1,4-2a\})$ the Lozi map $f$ has chaotic dynamics on a compact region in phase space: when additionally $b \in (0, \sqrt{2}(a - \sqrt{2}))$ this has a single mixing SRB measure \cite[Theorem~5]{Misiurewicz80}, all unstable manifolds have positive length when $b \in (0, a - \sqrt{2})$ \cite[Theorem]{Young85}. A Lozi attractor is shown in Figure~\ref{fig:CondMeasure}.
	
	Lozi maps are continuous, with a jump in the Jacobian across the singularity set $\S = \{x = 0\}$. 
	
	In a similar fashion to those we defined for the baker's map (Proposition~\ref{p:ConditionalMeasures}), conditional measures of the SRB measure $\dd\rho(x,y)$ on sets $\ell_{x_0} := \{x=x_0\}$ are well-defined for all $x_0 \in \R$ intersecting the support of the Lozi attractor \cite[Theorem~2.1
	]{lozi}. Let us notate these conditional measures as $\dd\rho(y \mid x_0)$. These conditional measures can be expected to have Hausdorff dimension strictly between $0$ and $1$: in particular, they lack any manifold structure. A histogram of $\dd\rho(y \mid 0) = \dd\rho(y \mid (x,y) \in \S)$ is plotted in Figure~\ref{fig:CondMeasure}: linear response for the Lozi map is determined from the mixing properties of the conditional measure on $\S$ \cite{lozi}, so we will be most interested in this particular conditional measure.

	We specifically conjecture that measures $\rho(\cdot\mid x_0)$, when pushed forward under the Lozi map $f$, converge back to the full SRB measure $\rho$, and that this convergence happens at an exponential rate:
	\begin{conjecture}\label{c:Reduced}
		For generic Lozi parameters $(a,b)$ and Lebesgue-almost all $x_0 \in \R$, the Lozi map has conditional mixing with respect to level curves $x = x_0$ (i.e. the measures $\rho(\cdot \mid x_0)$.
	\end{conjecture}

	We have strong and direct numerical evidence in favour of this conjecture: Figure~\ref{fig:Conjecture2} shows exponential decay of the correlation
	\begin{equation} \int A\circ T^n\,\dd\rho(\cdot\mid 0) - \int A\,\dd\rho \int \dd\rho(\cdot \mid 0) \label{eq:ReducedConjecture}\end{equation}
	by four orders of magnitude, with reliable error quantification. In fact, it seems that $A$ need only be piecewise $C^\Aord$. The consequence of Conjecture~\ref{c:Reduced} holding on the singular line $x=0$ (up to a technical generalisation to second-order mixing), we will show in \cite[Theorem~2.3
	]{lozi}, is that the Lozi map a formal linear response to bounded dynamical perturbations \cite{lozi}.

	\begin{figure}[t]
		\centering
		\includegraphics{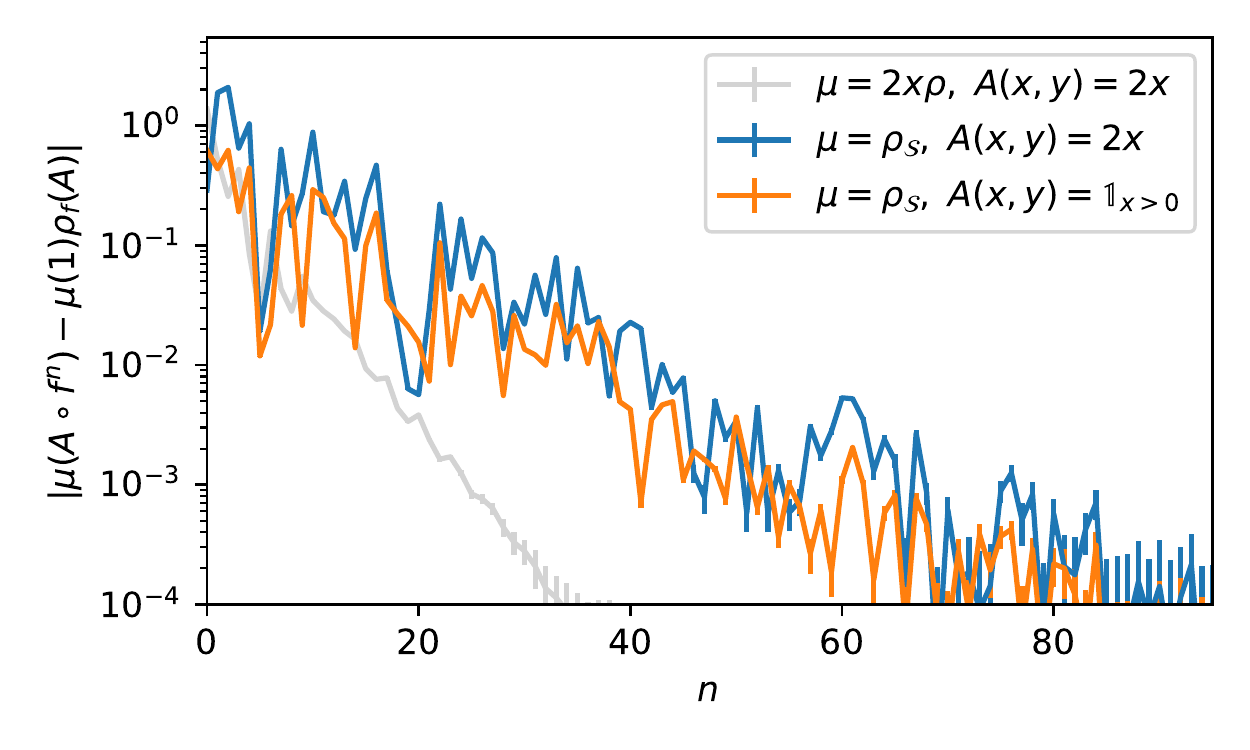}
		\caption{In orange and blue, for the parameters of the Lozi map $a = 1.8$ and $b = 0.35$, $|\condmeas(A \circ f^n) - \condmeas(1) \rho(A)|$ plotted for varying $n$ and two different observables $A$, where $\condmeas = \rho(\cdot\mid 0)$. The measure was sampled using $100$ time series of $10^7$ iterates of the segment dynamics $\vec f$, implemented in high-precision (196 bit) validated interval arithmetic. Details of the algorithm are given in Appendix~\ref{s:Algorithms}.
			Two forms of error are plotted: with error bars, 99\% confidence intervals for the sampling error, and with block error (not visible), the error arising from the interval arithmetic. 
		In grey, decay of autocorrelations for a smooth observable $A(x) = 2x$ against the SRB measure, using $100$ time series of $10^8$ iterates of the full dynamics $f$.
		}
		\label{fig:Conjecture2}
	\end{figure}
	
	It should be noted that obtaining valid samples of the quantities in \eqref{eq:ReducedConjecture} is tricky. To begin with, we are sampling from the conditional measure $\rho(\cdot\mid x_0)$, which is a codimension-one object on the attractor. We then must iterate forward under the Lozi map, which is chaotic and thus unstable. To validly perform this sampling rigorously, and quantify the associated numerical error, we have developed novel algorithms, presented in Appendix~\ref{s:Algorithms}. These algorithms could, we imagine, with care be extended to general hyperbolic dynamics.

\section{Conclusion}\label{s:Conclusion}

While in a numerical example, one can only consider a particular case, our results on baker's maps suggest that conditional mixing is a very robust property. Indeed, Theorem~\ref{t:Baker} shows that for an open dense set of such baker's maps (perhaps all baker's maps), conditional mixing holds on an open dense set of analytic curves. While these skew product maps are rather special, it seems from the Fourier dimension results that conditional mixing occurs when structure is {\it broken}, and so the situation might yet be better for conditional mixing in more general maps. To this end, we make the following conjecture:
\begin{conjecture}
	For all analytic Anosov diffeomorphisms on compact surfaces, conditional mixing holds on an open dense set of functions $H$ with no critical points on their zero level set.
\end{conjecture}

On the other hand, it is interesting that the exponential rates of decay for the conditional measures are substantially slower than for smooth observables against the full SRB measure: for example, in Figure~\ref{fig:Conjecture2} the rate of exponential convergence for a smooth observable $A(x,y) = x$ is much slower when the initial measure $\mu$ is a conditional measure rather than for the full SRB measure $\mu \sim \rho$. Furthermore, it appears that at least for the baker's map in some cases that this decay rate is independent of the conditioning submanifold (see discussion in Section~\ref{s:Baker}). In a related fashion, mixing rates against the SRB measure, which depend on the essential spectrum of the transfer operator in relevant function spaces, tend to be slower for smaller H\"older regularities of observables \cite{Baladi92}: thinking of the conditional measure $\rho(x\mid H(x) = 0)$ as equivalent to the SRB measure multiplied by a distribution $\delta(H(x))$ provides some connection between these two phenomena. Therefore, while the main connection we have seen appears to be to Fourier dimension, perhaps an appropriate functional analytic approach could yield fruit in studying this basic property of a chaotic system.

One might also ask what happens when the codimension of the conditioning submanifold is increased (in our study it has always been one). This is natural for the Bayesian filter problem since one typically makes repeated observations when observing a system: we expect to end up with, say, a vector of one-dimensional observations $y = (H(x), H(f(x)), H(f^2(x)),\ldots, H(f^m(x)))$.  If this dimension $m$ is no greater than the unstable dimension (i.e. the number of positive Lyapunov exponents), then the construction of the conditional measure as in Proposition~\ref{p:ConditionalMeasures} and \cite[Theorem~2.1
]{lozi} will go through, and we might feel empowered to say we expect conditional mixing to hold generically. On the other hand, if $m$ is more than twice the box-counting dimension of the attractor, then the probabilistic Takens embedding theorem would tells us that $y$ specifies $x$ exactly $\rho$-almost surely \cite{Baranski20}. In the intermediate case where $M$ lies between the unstable dimension and the attractor dimension, it could be that some kind of generic intersection property \`a la blenders \cite{Bonatti16} holds to give a conditional set of positive fractal dimension, which could also allow for some kind of conditional mixing. Numerical study in higher-dimensional systems may shed light on the situation.

\appendix

\section{Simulation of Lozi map dynamics}\label{s:Algorithms}

Rather than attempt to compute deterministic estimates of these systems we will proceed by Birkhoff--Monte Carlo sampling of the quantities we are interested in. Because we need to sample from measures $\rho(\cdot\mid x_0)$ conditioned on a codimension-1 manifold, it will be necessary to simulate not point dynamics but dynamics on sets of higher dimensions, the natural choice being local unstable manifolds. Helpfully, because the Lozi map is piecewise affine, local unstable manifolds are straight line segments. A ``segment dynamics'' is proposed in Section~\ref{ss:Segment} and a numerical implementation given in Section~\ref{ss:SegmentSim}.

However, we would like to be sure we are not merely percieving artifacts of sampling error or numerical imprecision: to achieve this, we will also need to quantify the statistical and deterministic errors associated with our numerical simulations. Section~\ref{ss:SegmentStable} gives, surprisingly, a stable algorithm to simulate the chaotic Lozi dynamics that is compatible with validated interval arithmetic, and Section~\ref{ss:CLT} explains the quantification of random sampling errors.

\subsection{Segment dynamics}\label{ss:Segment}

Let $\Lambda$ be the attractor of the Lozi map $f$, and for points $p$ whose orbits do not intersect the singular line $\S:= \{x = 0\}$,  define the local unstable manifold of a point $p \in \Lambda$ to be
\[ \Wul(p) := \left\{ q \in \Lambda: \lim_{n\to\infty} |f^{-n}(q) - f^{-n}(p)| = 0, \forall n \in \N^+\ \sigma_{f^{-n}(p)} = \sigma_{f^{-n}(q)}\}\right\}. \]
where $\sigma_{(x,y)} := \sign x$. These are segments of the full unstable manifolds which have always remained on the same side of the singular line $\S$.

Let $\vec{\mathcal{G}}$ be the set of directed open line segments in $\R^2$, i.e., open intervals where start and end points are distinguished. Then we can define a set of directed local unstable manifolds
\[ \vec \L = \left\{ \vec I \in \vec{\mathcal{G}} : \exists p \in \Lambda\ I = \Wul(p) \right\}, \]
which captures $\rho$-almost all local unstable manifolds, since Lozi maps are piecewise affine, and almost all unstable manifolds are of positive length.

Let us also define the following product space
\[ \vec \Lambda = \vec \L\times (0,1),\]
which we are going to use to parametrise each $\vec I \in \vec \L$. $\vec\Lambda$ is almost everywhere a two-to-one cover of $\Lambda$ by the map $\pi(\vec I_{p,q},t) = (1-t)p+tq$, where we denote the directed segment from point $p$ to point $q$ by $\vec I_{p,q}$.

As a result, up to a set of $\rho$-measure zero, we can lift the $f$-dynamics to $\vec \Lambda$, by a map of the form
\begin{equation} \vec f(\vec I, t) = (f(\vec I \cap \M_{\pi(\vec I,t)}), \lambda_{\vec I}(t)), \label{eq:VecFDynamics}\end{equation}
where $\lambda_{\vec I}: [0,1] \circlearrowleft$ is a full-branch expanding interval map. Let us define this a little more explicitly.

When $\vec I_{p,q} \cap \S$ is non-empty, then we know it has exactly one element which we denote by $s \in \S$ with $s = \pi(\vec I_{p,q},t_*)$ for some $t_* \in (0,1)$. The segment dynamics $\vec f$ can then be written explicitly as
\begin{equation}
\vec f(\vec I_{p,q},t) = \begin{cases} (\vec I_{f(p),f(q)},t), & \vec I_{p,q} \cap \S = \emptyset \\
(\vec I_{f(p),f(s)},t/t_*), & \vec I_{p,q} \cap \S \neq \emptyset \textrm{ and } t < t_* \\
(\vec I_{f(s),f(q)},(t-t_*)/(1-t_*)), & \vec I_{p,q} \cap \S \neq \emptyset \textrm{ and } t > t_*.
\end{cases}
\label{eq:VecFDynamicsExplicit}
\end{equation}

It turns out that the SRB measure $\rho$ also can be lifted to an invariant measure of $\vec f$ by
\[ \int_{\vec \Lambda} \Psi\, \dd\vec \rho = \int_\Lambda \frac{\sum_{(\vec I,t) \in \pi^{-1}(x)}\Psi(\vec I,t)}{2} \dd\rho(x), \]
because $\pi^{-1}$ is two-to-one $\rho$-almost everywhere we have that $(\vec \Lambda, \vec f, \vec \rho)$ has most two ergodic components (which are identical up to reversing the direction of the segments), and we will be able to sample $\vec\rho$ by iterating $\vec f$.

In defining a stable numerical method it will be useful to us that almost every point's local unstable manifold has endpoints originating from the singular line:
\begin{proposition}\label{p:CriticalOrbitSingularities}
	The set
	\[ \left\{(\vec I_{p,q},t) \in \vec\Lambda : p,q \in \cup_{n =1}^\infty f^n(\S) \right\}\]
	has full $\vec\rho$ measure.
\end{proposition}
This proposition is proved in Appendix \ref{a:CritOrbitSing}.

\subsection{Simulation of segment dynamics}\label{ss:SegmentSim}

Given our interval dynamics, we can sample the conditional measures using $\rho$ via the function $\kappa_{x_0}(\vec I) := \vec I \cap \{x = x_0\}$ and $\ell(\vec I) = |\vec I|$ the length of the segment. Because the SRB measure is uniformly distributed along unstable manifolds, we have that
\begin{equation} \int_{\vec\Lambda} A(x_0,y)\,\dd\rho(y\mid x_0) = \frac{\int_{\vec\Lambda} A\circ \kappa_{x_0}/\ell\, \dd\vec{\rho}}{\int_{\vec\Lambda} 1/\ell\, \dd\vec{\rho}} \label{eq:RhoSExpectation}
\end{equation}
assuming that no contribution to the sum is made when $\kappa_{x_0}(\vec I) = \emptyset$ (i.e, $\vec I$ does not intersect with the line we want to sample a conditional measure from). Then, for $\vec{\rho}$-almost-all starting values $(\vec I_0,t_0)$ we can estimate these expectations via a Birkhoff sum 
\begin{equation}
\int_{\vec\Lambda} \Psi\,\dd\vec\rho = \lim_{N\to\infty} \frac{1}{N} \sum_{n=0}^{N-1} \Psi(\vec f^n(\vec I_0,t_0))), \label{eq:RhoSBirkhoff}\end{equation}
with the segment dynamics that can be simulated using \eqref{eq:VecFDynamicsExplicit}. The convergence \eqref{eq:RhoSBirkhoff} for the $\Psi$ we are interested in uses the Birkhoff ergodic theorem and the fact that $\kappa_{x_0}, \ell$ do not depend on the direction of the interval (so the ergodic component of $\vec\Lambda$ we sample from is immaterial).

Because from \eqref{eq:VecFDynamicsExplicit} the $t$ dynamics are generated by full-branch interval maps that preserve Lebesgue measure, the branch dynamics are Markovian, with transition probabilities that are explicitly given. Because of this, the random dynamics
\begin{equation}
\vec f(\vec I_{p_{n+1},q_{n+1}},T_{n+1}) = \begin{cases} (\vec I_{f(p_n),f(q_n)},T_{n}), & \vec I_{p_n,q_n} \cap \S = \emptyset \\
(\vec I_{f(p_n),f(s_n)},T_{n}) \textrm{ with probability } t_*, & \vec I_{p_n,q_n} \cap \S \neq \emptyset \\
(\vec I_{f(s_n),f(q_n)},T_{n}) \textrm{ with probability } 1- t_*, & \vec I_{p_n,q_n} \cap \S \neq \emptyset.
\end{cases}
\label{eq:VecFDynamicsImplicit}
\end{equation} 
generate the $\vec f$ dynamics at equilibrium, where the $T_n$ are uniformly distributed, dependent hidden variables which are given by
\[ T_{n+1} = \lambda_{\vec{I}_{p_n,q_n}}(T_n).\] 
To sample that $\vec I$ we do not actually need to know the $T_n$, but they can be reconstructed from a time series of the segments by sampling the final value $T_{n_{\rm final}} \sim \mathrm{Uniform}(0,1)$ and iterating backwards according to (\ref{eq:VecFDynamicsExplicit}): inverting the $\lambda_{\vec I}$ yields a contraction. This means we do not have to directly simulate an expanding map (which would have been problematic for rigorously validated simulation of the dynamics). In practice this is an very effective way to simulate the segment dynamics.

\subsection{Validated numerical implementation of segment dynamics}\label{ss:SegmentStable}

However, computers can only encode real numbers to finite precision. Thus, at every computational step the results must be rounded to a given tolerance, introducing small errors, which may invalidate fine numerical results such as we wish to obtain. Validated interval arithmetic provides a vehicle to quantify the errors, but to use it we must first produce a deterministically stable algorithm to simulate a generic long chaotic time series. In this subappendix will we present such an algorithm, introducing first the notion of interval arithmetic.

Let the set of closed intervals in $\R$ be $\mathcal{I}$. The idea of validated interval arithmetic is to represent real numbers $\alpha$ by an interval $\mathfrak{a} \in \mathcal{I}$ such that we know $\alpha \in \mathfrak{a} \subset \R$. Such an interval $\mathfrak{a}$ is given by its upper and lower bounds, and we can restrict the set of allowed intervals so these upper and lower bounds are representable in the finite precision computer encoding. A function $g: \R^d \to \R^e$ can be implemented in validated arithmetic through a function $\mathfrak{g}: \mathcal{I}^d \to \mathcal{I}^e$ such that $\mathfrak{g}(\mathfrak{a})$ will always contain $g(\mathfrak{a})$. In such a way one can be absolutely certain that $\alpha$ is contained in a set $\mathfrak{g}(\mathfrak{a})$ and so on.


The unstable manifolds $\vec I_{p,q}$ are defined by their endpoints, which update under the chaotic dynamics $f$. By definition, $f$ is exponentially stretching, making it very difficult in general to obtain a rigorously validated time series. However, the fact that we are constantly resetting the segment endpoints to the critical line (\ref{eq:VecFDynamicsExplicit}) makes efficiently obtaining such a time series quite possible.

We will store our segments $\vec I_{p,q}$ as $Q \vec I_{p',q'}$, where $Q$ is an orthogonal transformation of $\R^2$ (i.e. a rotation matrix), and $p',q' \in \R^2$ have identical second coordinate. (Of course, an segment can be stored as a $2\times 2$ matrix of its endpoints' coordinates.) Thus, $Q$ rotates the phase space so that the unstable direction on the segment is along the first coordinate. If our segment at the next step is $\vec I_{p_1,q_1} = Q_1 \vec I_{p_1',q_1'}$, we do not compute the quantities on the right-hand side from $f$ explicitly, but rather, since we have on $f^{-1} (\vec I_{p_1,q_1})$ that the Lozi map is affine, having for some $J$ that
\[ f(x) = J x + e_1, \]
we make the QR decomposition
\begin{equation} Q_1 R_1 = J Q, \label{eq:QR}\end{equation}
where $Q_1$ is a rotation matrix and $R_1$ is upper triangular, and set 
\[ \vec I_{p_1',q_1'} = R_1 \vec I_{p',q'} + Q_1^\top e_1. \]
This means the dynamics in the endpoints' shared second coordinate is contracting (and thus numerically stable), as in fact are the dynamics of $Q$, and that except for the (discrete) choice of branch $\M_{\pm}$, these are both independent of the points' first coordinates. This means the second coordinate as well as the rotation matrix $Q$ can be stably approximated in interval arithmetic\footnote{It is however necessary to explicitly code the QR decomposition (\ref{eq:QR}) in a way optimised for interval arithmetic, as many standard {\tt qr} routines give sub-par interval bounds that will lead to numerical blow-up of the algorithm.}.

The first coordinates of $p'$ and $q'$, on the other hand, have expanding dynamics. However, when the segment $\vec I_{p,q}$ is cut by the singular line at a point $r$, we can define $r' = Q'r$ without reference to these first coordinates. In particular, the singular line in the transformed coordinates is $Q^\top \S$ which solves some equation $x' = \beta y'$, with $\beta$ bounded because unstable manifolds are uniformly transversal to the singular line \cite{Young85}. Then, if the shared second coordinate of $p', q'$ is $y'$, then we can write 
\[ s' = (\beta y',y'). \]
Notably, this point $s'$ is generated using only quantities whose numerical error remains stable. It thus replaces either one of $p, q$ which contain dynamics where the error grows. By Proposition~\ref{p:CriticalOrbitSingularities}, almost every $p, q$ will eventually be replaced by such an $s$, resetting the size of its error and therefore ensuring it does not grow too big. 

Implementation of the algorithm described above in validated interval arithmetic is straightforward, because it is stable, but here we must be careful: when our local unstable manifold is split and the choice of child manifold is to be made, the $t_*$ used to determine the choice is interval-valued (i.e. in this set $\mathcal{I}$ and likely of positive width). The natural way to choose the branch to continue with is made by sampling a uniform random variable $U$ and comparing it with $t_*$: the choice of segment is then clear except where there is an overlap between $U$ and $t_*$. A simple way to deal with this problem is to choose the floating-point precision small enough to make an overlap unlikely enough to invite references to the age of the universe (twice the bits of the standard double-precision is enough). More comprehensive handling of this eschatological edge-case may be done in various ways, including using importance sampling on multiple time series\footnote{One compares $U$ with some real number $t_{**} \in t_*$ and reweights the time series by $t_*/t_{**}$ or $(1-t_*)/(1-t_{**})$ as appropriate: note that the weights are also interval-valued.} 

We will therefore be able to find an interval hypercube containing an exact time series $\{\vec f^n(\vec I,t_0)\}_{n = 0,\ldots,N}$ from (\ref{eq:VecFDynamicsImplicit}), where $t_0$ is implicitly defined by the random selection when the segment is cut.


\subsection{Quantification of statistical error}\label{ss:CLT}

As the $\vec f$ dynamics we sample is just a two-to-one lift from the $f$ dynamics, which have a spectral gap \cite{BaladiGouezel09}, we expect that for large $N$ the error between Birkhoff means and true expectations \eqref{eq:RhoSBirkhoff} obey a central limit theorem \cite{Chernov95}. If we have several long time series
\[ \Psi^{(r)} = \frac{1}{N} \sum_{n=0}^{N-1} \Psi(\vec f^n(\vec I^{(r)}))) \]
for $R$ independent samples from $I^{(r)} \sim \vec \rho$, then
for sufficiently large $N$, the sample mean $\bar\Psi$ of the $\Psi^{(r)}$ will have expectation $\vec\rho(\Psi)$ (from the initialisation of the time series at equilibrium), and will differ from this by a factor of $\O(1/\sqrt{RN})$. Helpfully, we can quantify this deviation {\it a posteriori}: using the Gaussian behaviour of the $\Psi^{(r)}$ we have that if $s_\Psi^2$ is the sample variance, then for large $N$ we have
\[ \frac{\bar\Psi - \vec\rho(\Psi)}{\sqrt{R} s_\Psi} \sim t_{R-1}, \]
where $t_{R-1}$ is a Student $t$ distribution with $R-1$ degrees of freedom. This allows us to put confidence intervals on our estimates, as in Figure~\ref{fig:Conjecture2}. 
Such a principle has been used to test for linear response in previous work \cite{Gottwald16}.

To obtain an accurate sample from $\vec \rho$ in initialising our time series we begin by initialising $(\vec I_{p_0,f(p_0)},t^{(r)}) \in \vec\Lambda$, where $p_0 = (\tfrac{2}{2 + a - a^2 + 4b},0)$ and $t^{(r)} \sim \textrm{Uniform}(0,1)$ (in fact, implicitly using the above random choice methods). This initial measure lies in the Banach space that converges exponentially quickly to the physical measure \cite{DemersLiverani08} and so by making the spin-up time $n_{init}$ sufficiently long, we can ensure that our sampling initialisations $\vec I^{(r)} = \vec f^{n_{init}}(\vec I_{p_0,p_1},t^{(r)})$ come from a distribution exponentially close to $\vec \rho$.

	\section{Proof of baker's map result}\label{s:Baker}

In this appendix we will prove Theorem~\ref{t:Baker} on exponential conditional mixing for baker's maps of the form \eqref{eq:BakerMap}. 

For concision when quantitatively referring to Fourier dimension, let us say that a measure-function pair $(\nu,\psi)$ has $(\eta,C)$ Fourier decay if for all $j \in \Z \backslash \{0\}$,
\begin{equation} \left| \int_{0}^1 e^{2\pi i j \psi(x)} \,\dd\nu(x) \right| \leq C |j|^{-\eta}, \label{eq:FourierDecay} \end{equation}
and $\int_0^1 |\dd\nu| \leq C$. This implies that the Fourier dimension of $\psi_*\nu$ is at least equal to $2\eta$. 

Fourier decay is invariant under translations of $\psi$:
\begin{proposition}\label{p:FourierDecayTranslation}
	Suppose that $(\nu,\psi)$ has $(C,\eta)$ Fourier decay. Then for all $t \in \R$, so does $(\nu,\psi + t)$.
\end{proposition}
\begin{proof}
	We have 
	\[ \left| \int_{0}^1 e^{2\pi i j (\psi(x) + t)} \,\dd\nu(x) \right| = \left|\int_{0}^1 e^{2\pi i j \psi(x)} \,\dd\nu(x) \right| \leq C |j|^{-\eta}, \]
	and the integral of $|\nu|$ remains no greater than $C$, as required.
\end{proof}
	
To prove Theorem~\ref{t:Baker} we will employ a separate Fourier dimension theorem \cite{Mosquera18,Sahlsten20} for each of the conditions in the theorem's statement. The common component is the following lemma (into which any new Fourier dimension results may also be substituted):

\begin{lemma}\label{l:BakerLemma}
	Suppose one has a modified baker's map $b$ with contracting maps $v_i$. Let $\nu_0$ be the probability measure such that $v_i^* \nu_0 = k^{-1} \nu_0$ for all $i = 1,\ldots, k$, and let $\psi \in C^2$ be such that $\psi' \neq 0$ and $(\nu_0,\psi)$ has $(C,\eta)$ Fourier decay. 
	
	Let $\gamma \in (1-\eta,1)$, $\beta \in (0,1]$ and $\alpha \in (2-\eta-\gamma,1)$. 
	
	Then there exists $\xi \in (0,1)$ and $C'$ depending only on $C,\eta,\alpha,\beta\gamma,\psi'$ such that for all $A \in C^{\alpha;\beta}$ and $B \in C^{\gamma}$
	\[ | \rho_0(A\circ b^n\, B) - \rho_0(B) \rho(A) | \leq C' \| A \|_{C^{\alpha;\beta}} \|B\|_{C^\gamma} \xi^n \]
	where $\rho_0$ is defined as in Theorem~\ref{t:Baker}.
\end{lemma}

\begin{proof}[Proof of Theorem~\ref{t:Baker}]
	The measures $\nu_0$ from Lemma \ref{l:BakerLemma} are the measures of maximal entropy of these expanding iterated function schemes: in particular, they are Gibbs (with constant weights) and atomless. If we have that $(\nu_0,\psi)$ has Fourier decay then so do $(\nu_0,\psi)$ uniformly from Proposition~\ref{p:FourierDecayTranslation}, and Lemma~\ref{l:BakerLemma} then secures us the theorem. Since $\nu_0$ is a probability measure, it is only necessary to check that \eqref{eq:FourierDecay} holds, which we do procedurally from existing results.
	\\
	
	That \eqref{eq:FourierDecay} holds for option \ref{opt:3} is a simple application of \cite[Theorem~3.1]{Mosquera18} (and in fact here, $\eta$ is independent of $\psi$).
	\\
	
	To see it for options \ref{opt:1} and \ref{opt:2} requires a little more cunning. We have that $\psi$ is a diffeomorphism onto its image: let $\omega(x) = \omega_0 + \omega_1 x$ map $\psi([0,1])$ linearly onto $[0,1]$, so that $\tilde \psi = \omega \circ \psi: [0,1]\circlearrowleft$ is a diffeomorphism. If $\{v_{\mathbf{i}}\}_{\mathbf{i} \in \{1,\ldots, k\}^n}$ are $n$-fold compositions of the contractions $v_{i}$ then for some large enough $n$, the $n$-fold compositions $\{\tilde\psi \circ v_{\mathbf{i}} \circ \tilde\psi^{-1}\}_{\mathbf{i} \in \{1,\ldots, k\}^n}$ are uniformly contracting. They are also totally nonlinear and $C^2$ with bounded distortion. If \ref{opt:1} holds then their ranges fill $[0,1]$ and if \ref{opt:2} holds then they are analytic. Furthermore, under either option they remain totally nonlinear. By \cite[Theorem~1.1]{Sahlsten20} their measure of maximal entropy $\tilde \nu_0$ (which is Gibbs and atomless) therefore has polynomial decay of its Fourier transform, that is, for all $l \in \R\backslash\{0\}$,
	\[ \left|\int_0^1 e^{-2\pi i l y} \dd\tilde \nu_0(y)\right| \leq C |l|^{-\eta} \]
	for some $\eta>0$ and $C<\infty$.
	
	Now, this measure $\tilde \nu_0$ is also the measure of maximal entropy of the conjugated single iterates $\{ \tilde\psi \circ v_i \circ \tilde\psi^{-1} \}_{i = 1,\ldots,k}$; from the conjugacy we therefore know that $\tilde \nu_0 = \tilde\psi^* \nu_0$. Hence,
	
	\begin{align*}
		\int_0^1 e^{-2\pi i l y} \dd\tilde \nu_0(y) &= \int_0^1 e^{-2\pi i l \omega(\psi(y))} \dd \nu_0(y)\\
		&= e^{-2\pi i l \omega_0}  \int_0^1 e^{-2\pi i l \omega_1 \psi(y)} \dd \nu_0(y),
	\end{align*}
	so setting $l = j/\omega_1$ we obtain that
	\[ \left| \int_0^1 e^{-2\pi i j \omega_1 \psi(y)} \dd \nu_0(y) \right| \leq C |\omega_1|^{\eta} |j|^{-\eta}, \]
	as required.
	%
\end{proof}

The relevance of $\nu_0$ is because it is the cross-section of the SRB measure along lines of constant $x$ (i.e. local stable manifolds):

\begin{proposition}\label{p:SRBMeasure}
	Let $\nu_0$ be as in Lemma \ref{l:BakerLemma}. Then $\rho = \Leb \times \nu_0$ is the SRB measure of $b$.
\end{proposition}

Henceforth we will find it useful to notate the unstable dynamics $\kappa(x) = kx \mod 1$.

\begin{proof}[Proof of Proposition~\ref{p:SRBMeasure}]
	$\rho$ is conditionally absolutely continuous along unstable manifolds (which are lines of fixed $y$), and solves
	\begin{align*} 
	\int_D A\circ b\, \dd\rho &= \int_0^1 \int_0^1 A(\kappa(x), v_{\lceil kx \rceil}(y))\, \dd \nu_0(y)\,\dd x\\
	&= \sum_{i=1}^k \int_{(i-1)/k}^{i/k} \int_0^1 A(\kappa(x), v_{i}(y))\, \dd \nu_0(y)\,\dd x\\ 
	&= \sum_{i=1}^k \int_{0}^1 k^{-1} \int_0^1 A(x, v_{i}(y))\, \dd \nu_0(y)\,\dd x \\
	&= \sum_{i=1}^k \int_{0}^1 k^{-1} \int_0^1 A(x, y)\, \dd (v_i^*\nu_0)(y)\,\dd x \\
	&= \int_0^1 \int_0^1 A(x,y) k^{-1} \sum_{i=1}^k \dd(v_i^*\nu_0)(y)\,\dd x\\
	&= \int_0^1 \int_0^1 A(x, y)\, \dd \nu_0(y)\,\dd x = \int_D A\,\dd\rho.
	\end{align*}
	Hence, it is an SRB measure. 
\end{proof}

This allows us to prove the existence of our conditional measures $\rho_t$:
\begin{proof}[Proof of Proposition~\ref{p:ConditionalMeasures}]
	To prove \eqref{eq:BakerSlice}, noticing that $\rho$ is just a product measure of the uniform measure in $x$ (therefore $t$) and $\nu_0$ in $y$, we have that
	\[ \frac{1}{2\delta}\int_{\left\{\Psi(s,y)\mid |s-t|<\delta,y\in[0,1]\right\}} A(x,y)\,\dd\rho =  \int_{[0,1]} \frac{1}{2\delta}\int_{[s-\delta,s+\delta]}A(\psi(y)-t,y)\,\dd t\,\dd\nu_0(y), \]
	where $\nu_0$ is defined in Lemma~\ref{l:BakerLemma}.
	This integral is absolutely bounded by $\sup |A|$, and so by the dominated convergence theorem
	\begin{align*} \lim_{\delta\to 0} \frac{1}{2\delta}\int_{\left\{\Psi(s,y)\mid |s-t|<\delta,y\in[0,1]\right\}} A(x,y)\,\dd\rho &= \int_{[0,1]} \lim_{\delta \to 0} \frac{1}{2\delta}\int_{[t-\delta,t+\delta]}A(\psi(y)-s,y)\,\dd s\,\dd\nu_0(y) \\
	&= \int_{[0,1]} A(\psi(y)-t,y)\,\dd \nu_0(y)\\
	& = \int_{D} A\,\dd(\Psi(t;\cdot)_* \nu_0). \end{align*}
	
	This means we must define
	\begin{equation}\label{eq:CondMeasDef}
	\rho_t := \Psi(t;\cdot)_* \nu_0,
	\end{equation}
	and so get \eqref{eq:BakerSlice}.
	 
	To prove the second part, we have that for any $E \subseteq D_\Psi$,
	\[ \rho_t(E) = \int_{[0,1]} \mathbb{1}_E(\Psi(t;y))\,\dd \nu_0(y), \]
	so
	\[ \int_{-t_*}^{t_*} \rho_t(E)\,\dd t = \int_{-t_*}^{t_*}\int_{[0,1]} \mathbb{1}_E(\Psi(t;y))\,\dd \nu_0(y)\,\dd t. \]
	By a change of coordinates $(x,y) = \Psi(t,y) = (\psi(y)-t,y)$, and using Proposition~\ref{p:SRBMeasure}, we have
	\[ \int_{-t_*}^{t_*} \rho_t(E)\,\dd t = \int_{D_\Psi} \mathbb{1}_E(x,y)\,\dd \rho(x,y) = \rho(E). \]

\end{proof}

The following technical lemmas will be of use in following proofs. We will prove them in Appendix~\ref{a:baker}.
\begin{lemma}\label{l:DCT}
	Suppose that for some $\alpha \in (0,1]$, $\phi: [0,1] \to \R$ is a piecewise $\alpha$-H\"older function with a finite number of jumps, and $\nu$ is an integrable atomless measure. Let $\hat \phi_j, \hat \nu_j$ be the respective Fourier coefficients of $\phi$ and $\nu$. Then
	\begin{equation} \int \phi \,\dd\nu = \sum_{j \in \mathbb{Z}} \hat \phi_{-j} \hat \nu_j, \label{eq:IntegralFourierDecomposition}\end{equation}
	provided the sum is absolutely convergent.
\end{lemma}

For $\alpha \in (0,1]$ let the H\"older semi-norm on a set $E \subseteq [0,1]$ be defined as follows:
\begin{equation} |\phi|_{C^\alpha(E)} := \sup_{[x,y] \subseteq E} \frac{|\phi(x)-\phi(y)|}{|x-y|^\alpha}. \label{eq:PiecewiseHolderSemiNormDef}\end{equation}

\begin{lemma}\label{l:HolderBVFourierDecay}
	Suppose $\phi:[0,1] \to \R$ is piecewise $\alpha$-H\"older with jumps on a set $S \subset (0,1]$ (including at $1$ if it is not periodic). Then for $j \in \Z \backslash\{0\}$,
	\[ \left| \int_0^1 \phi(x) e^{-2\pi i j x}\,\dd x \right| \leq |\phi|_{C^\alpha(S^c)} |j|^{-\alpha} + |S| \|\phi \|_{L^\infty} |j|^{-1}. \]
\end{lemma}

With these lemmas in hand, we will try and prove exponential conditional mixing in the projection of the baker's map onto the $x$ coordinate. This next lemma is the heart of the proof
\begin{lemma}\label{l:Prop4}
	Suppose $\phi: \R/\Z$ is as in Lemma \ref{l:HolderBVFourierDecay}, and $\nu$ is an integrable atomless measure with Fourier coefficients $\hat \nu_j$, $(\id,\nu)$ has $(C_\nu, \eta)$ Fourier decay, and $\alpha > 1 - \eta$. Then
	\[ \left| \int_0^1 \phi \circ \kappa^n \,\dd\nu - \int_0^1 \phi\,\dd x \int_0^1 \dd\nu \right| \leq \frac{4C_\nu}{\alpha+\eta-1} k^{-n\eta} (|\phi|_{C^\alpha(S^c)} + |S| \|\phi\|_{L^\infty}). \]
\end{lemma}
\begin{proof}
	Recall that $\kappa^n(x) = k^n x\mod 1$. The Fourier coefficients of $\phi \circ \kappa^n$ are zero except for those whose indices are multiples of $k^n$:
	\[ \int_0^1 \phi(\kappa_n(x))e^{-2\pi ij k^n x}\,\dd x = \hat \phi_{k^n j}. \]
	These decay as $\O(|j|^{-\alpha})$, whereas the Fourier coefficients of $\nu$ are $\O(|j|^{-\eta})$, so we know their convolution is summable. By Lemma \ref{l:DCT} we therefore have
	\[ \int_0^1 \phi \circ \kappa^n \,\dd\nu = \sum_{j\in \Z} \hat \phi_{-j} \hat v_{k^n j}. \]
	This means
	\begin{align*}
	\left| \int_0^1 \phi \circ \kappa^n \,\dd\nu - \hat\phi_0 \hat\nu_0 \right| &\leq \sum_{j=1}^\infty |\hat \phi_{-j}| |\hat v_{k^n j}| + |\hat \phi_{j}| |\hat v_{-k^n j}| \\
	&\leq \sum_{j=1}^\infty 2\left( |\phi|_{C^\alpha(S^c)} |j|^{-\alpha} + |S| \|\phi \|_{L^\infty} |j|^{-1} \right) C_\nu |k^n j|^{-\eta} \\
	&\leq 2C_\nu k^{-n\eta} \left( \frac{\alpha+\eta}{\alpha + \eta - 1} |\phi|_{C^\alpha(S^c)} + \frac{1+\eta}{\eta} |S| \|\phi \|_{L^\infty} \right).
	\end{align*} 
	Elementary inequalities on the fractions, and the zeroth Fourier coefficient's definition as the total integral give the required result.
\end{proof}

We now attempt to connect this one-dimensional picture in $\kappa$ to the two-dimensional picture of the baker's map. In this proposition we define a one-dimensional observable $A_{m,y_0}(x)$ that in the following proposition we find will closely approximate $A(b^m(x,y))$ for any $y$, when $m$ is large enough.
\begin{lemma}\label{l:Prop5}
	Suppose that $\nu, \alpha$ are as in Lemma~\ref{l:Prop4}. Suppose that $A: D \to \R$ has $|A|_{\alpha,x} < \infty$ and let 
	\[ A_{m,y_0}(x) := A(b^m(x,y_0)).\]
	Then
	\[ \left| \int_0^1 A_{m,y_0} \circ \kappa^n \,\dd\nu - \int_0^1 A_{m,y_0}\,\dd x \int_0^1 \dd\nu \right| \leq \frac{4C_\nu}{\alpha+\eta-1} k^{m-n\eta} (|A|_{\alpha,x} + \|A\|_{L^\infty}). \]
\end{lemma}
\begin{proof}
	It is clear that $A_{m,y_0}$ is piecewise $\alpha$-H\"older with jumps at $S_m := \{i/k^m : i = 1,\ldots, k^m\}$.
	
	We can also bound its H\"older constant. Suppose $[x,z] \subset (0,1] \backslash S_m$. This means that $b^l(x,y_0)$ and $b^l(x,z_0)$ lie on the same piece of $b$ for all $0 \leq l < m$, and therefore that $b^m(x,y_0)$ and $b^m(z,y_0)$ have the same $z$ component. As a result,
	\begin{align*}
	|A_{m,y_0}(x)- A_{m,y_0}(z)| &\leq |A(b^m(x,y_0)) - A(\kappa^m(z),y_0)|\\
	&\leq |A|_{\alpha,x} |\kappa^m(x) - \kappa^m(z)|^\alpha\\
	&\leq |A|_{\alpha,x} k^{m\alpha} |x-z|^\alpha.
	\end{align*}
	From \eqref{eq:PiecewiseHolderSemiNormDef}, this means that $|A_{m,y_0}|_{C^\alpha(S_m^c)} \leq |A|_{\alpha,x} k^{m\alpha}$.
	
	Applying Lemma~\ref{l:Prop4}, we get that 
	\begin{align*}
	\left| \int_0^1 A_{m,y_0} \circ \kappa^n \,\dd\nu - \int_0^1 A_{m,y_0}\,\dd x \int_0^1 \dd\nu \right| &\leq \frac{4C_\nu}{\alpha+\eta-1} |k|^{-n\eta} (|A_{m,y_0}|_{C^\alpha(S_m^c)} + |S_m| \|A_{m,y_0}\|_{L^\infty})\\
	&\leq \frac{4C_\nu}{\alpha+\eta-1} k^{-n\eta} (|A|_{\alpha,x} k^{m\alpha} + k^m \|A\|_{L^\infty})\\
	&\leq \frac{4C_\nu}{\alpha+\eta-1} k^{m-n\eta} \left(|A|_{\alpha,x} + \|A\|_{L^\infty}\right),
	\end{align*}
	as required.
\end{proof}

\begin{proposition} \label{p:Prop6}
	For all $\beta > 0$, $m \in \N$, $A$ with finite $|\cdot|_{\beta,y}$ norm, and $x,y,y_0 \in [0,1]$,
	\[ | A_{m,y_0}(x) - (A\circ b^m)(x,y) | \leq \mu^{\beta m} |A|_{\beta,y} |y - y_0|^\beta.\]
\end{proposition}
\begin{proof}
	We prove this by induction on $m$. We have that
	\[ A_{0,y_0}(x) - A(x,y) = A(x,y_0) - A(x,y), \]
	which is bounded for all $y,y_0 \in [0,1]$ by $|A|_{\beta,y} |y - y_0|^\beta$. Suppose then that our proposition holds for some $m$. Then for any $y \in [0,1]$,
	\[ b^{m+1}(x,y)= b^m(\kappa(x), \nu_i(y)), \]
	where the map branch $i = \lceil kx \rceil$. As a consequence, $A_{m+1,y_0}(x) = A_{m,\nu_i(y_0)}(\kappa(x))$, and so
	\begin{align*}
	| A_{m+1,y_0}(x) - (A\circ b^{m+1})(x,y) | &= A_{m,\nu_i(y_0)}(\kappa(x)) - (A\circ b^m)(\kappa(x), \nu_i(y))\\
	&\leq \mu^{\beta m} |A|_{\beta,y} |\nu_i(y) - \nu_i(y_0)|^\beta\\
	&\leq \mu^{\beta(m+1)} |A|_{\beta,y} |y - y_0|^\beta
	\end{align*}
	as required for the inductive step, where we used that the $\nu_i$ contract points by a factor of $\mu$.
\end{proof}

We can then put Lemma~\ref{l:Prop5} and Proposition~\ref{p:Prop6} together to prove a primitive version of Lemma \ref{l:BakerLemma}:
\begin{proposition} \label{p:Prop7}
	Suppose $\alpha + \eta > 1$. Then there exists $\xi < 1$ depending only on $\eta, k, \mu$ and $\beta$ and there also exists $C$ such that if $\psi$ is $C^1$ with $\psi'=0$ on a finite set, and $\nu$ is an atomless measure such that $(\nu,\psi)$ has $(\eta,C_{\nu,\psi})$ Fourier decay, then
	\begin{equation}\label{eq:Prop7}
	\left| \int_D (A \circ b^n)(\psi(y),y)\,\dd\nu(y) - \int A\,\dd\rho \int \dd\nu \right| \leq C C_{\nu,\psi} \xi^n \|A\|_{\alpha;\beta}. 
	\end{equation}
\end{proposition}
\begin{proof}
	We will divide $n = m + l$ and the difference in \eqref{eq:Prop7} up into several pieces that we will bound largely using previous results.
	
	To begin with, we have as an application of Proposition~\ref{p:Prop6} that for any $y_0$,
	\begin{equation}\label{eq:Prop7Piece1}
	\left|\int_D \left((A \circ b^{m+l})(\psi(y),y) - (A_{m,y_0} \circ b^l)(\kappa^n(\psi(y)))\right)\,\dd\nu(y)\right| \leq \mu^{\beta m} |A|_{\beta,y} \int |\dd \nu|.
	\end{equation}
	
	Now, $\psi_*\nu$ is atomless with $(\id,\psi_*\nu)$ having $(\eta,C_{\nu,\psi})$ Fourier decay, so we can apply Lemma~\ref{l:Prop5} to obtain
	\begin{equation}\label{eq:Prop7Piece2} 
	\left| \int_0^1 (A_{m,y_0} \circ b^l)(\kappa^n(\psi(y)))\,\dd\nu(y) - \int_0^1 A_{m,y_0}\,\dd x \int \dd\psi_*\nu \right| \leq \frac{4C_{\nu,\psi}}{\alpha + \eta - 1} k^{m - l\eta}(|A|_{\alpha, x} + \|A\|_{L^\infty}). 	
	\end{equation}
	
	Next,
	\begin{align*}
	\int_0^1 A_{m,y_0}\,\dd x &= \int_0^1 A(b^m(x,y_0))\,\dd x\\
	&= \int_D A(b^m(x,y_0))\, \dd\rho(x,y)
	\end{align*}
	because from Proposition~\ref{p:SRBMeasure}, the SRB measure $\rho$ projects to Lebesgue measure in the $x$ coordinate. With this, we have that
	\begin{align*}
	\left| \int_0^1 A_{m,y_0}\,\dd x - \int_D A\circ b^m\,\dd\rho \right| &= \left| \int_D (A_{m,y_0} - A\circ b^m)\,\dd\rho \right| \\
	&\leq \mu^{\beta m} |A|_{\beta,y}
	\end{align*}
	using Proposition~\ref{p:Prop6}. Recalling also that $\rho$ is $b$-invariant, and pushing $\nu$ forward preserves its total integral, we can say that 
	\begin{equation}\label{eq:Prop7Piece3}
	 \left| \int_0^1 A_{m,y_0}\,\dd x \int \dd \psi_*\nu - \int_D A\,\dd\rho \int_D \dd\nu \right| \leq \mu^{\beta m} |A|_{\beta,y} \int |\dd \nu|. 
	\end{equation}
	Putting \eqref{eq:Prop7Piece1}, \eqref{eq:Prop7Piece2} and \eqref{eq:Prop7Piece3} together, we have that
	\[ \left|\int_D (A \circ b^{m+l})(\psi(y),y)\,\dd\nu - \int_D A\,\dd\rho \int_D \dd\nu \right| \leq 2 \int |\dd\nu|\,\mu^{\beta m} |A|_{\beta, y} + \frac{4C_{\nu,\psi}}{\alpha + \eta - 1} k^{m - l\eta}(|A|_{\alpha, x} + \|A\|_{L^\infty}). \]
	By setting $l = \lceil(1 - \tfrac{\eta\log k}{\log \mu^{-1} \beta + (1+\eta) \log k}) n \rceil$ we obtain that there exists a constant $C$ depending on $\alpha, \beta, \eta, \mu, k$ such that \eqref{eq:Prop7} holds with 
	\[ \xi = k^{-\eta/(\beta + (1+\eta) \log k/\log \mu^{-1})}. \]
	\end{proof}

At this point, if we set $B \equiv 1$, we could prove Lemma \ref{l:BakerLemma} already. However, to incorporate it we need to show that we can multiply $\nu_0$ by sufficiently smooth H\"older functions and still retain adequate Fourier decay.

\begin{lemma}\label{l:Prop8}
	Suppose that $\psi:[0,1]\circlearrowleft$ is a $C^1$ diffeomorphism onto its image, and $\nu_0$ is an atomless measure such that $(\psi,\nu)$ has $(\eta, C_{\nu_0,\psi})$ Fourier decay.
	Then for all $\gamma \in (1-\eta,1]$ there exists $C$ such that for all $B \in C^\gamma(D)$,
	\[ \left| \int_0^1 e^{2\pi i j \psi(y)} B(y,\psi(y)) \,\dd\nu(y) \right| \leq C_{\nu_0,\psi} C \| B\|_{C^\gamma}|j|^{-(\eta+\gamma-1)},\]
	that is, if the measure $\dd\sigma(y) := B(y,\psi^{-1}(y))\dd\nu(y)$, then $(\sigma,\psi)$ has $(\eta+\gamma-1, C_{\nu,\psi} C)$ decay; furthermore $\sigma$ is atomless.
\end{lemma}
\begin{proof}
	We have that $\sigma$ is atomless because it is defined as an atomless measure multiplied by a bounded function.
	 
	Define function $B_\psi \in C^\gamma([0,1])$ such that $B_\psi(x) = B(\psi^{-1}(x),x)$ on $\psi([0,1])$. It is possible to do this so that 
	\[ |B_\psi|_{C^\gamma([0,1])} = |B_\psi|_{C^\gamma(\psi([0,1]))} \leq C'_{\psi'} |B|_{C^\gamma}, \]
	where $C'_{\psi'} = 1 + \|1/\psi'\|_{L^\infty} \geq 1$, and $\|B_\psi\|_{L^\infty} \leq \|B\|_{L^\infty}$.
	
	By Lemma \ref{l:HolderBVFourierDecay} we have that $\hat b_l$, the Fourier coefficients of $B_\psi$, have a certain bound,
	\[ |\hat b_l | \leq \begin{cases}\|B\|_{L^\infty},& l = 0\\ |l|^{-\gamma} C'_{\psi'} |B|_{C^\gamma}, & \textrm{else} \end{cases} \leq \min\{1,|l|\}^{-\gamma} C'_{\psi'} \|B\|_{C^\gamma} \]
	and so in particular for any $j$, the Fourier coefficients of $e^{2\pi i j\cdot} B_\psi$, which are just shifts of those of $B_\psi$, decay as $\O(l^{-\gamma})$.
	
	Therefore, we can apply Lemma \ref{l:DCT} to get that 
	\[ \int_0^1 e^{2\pi i j \psi(y)} B(y,\psi(y)) \,\dd\nu(y) = \sum_{l \in \Z} \hat b_{-j-l} \hat v_{l} \]
	and so
	\begin{align*} \left|\int_0^1 e^{2\pi i j \psi(y)} B(y,\psi(y)) \,\dd\nu(y)\right| &\leq C_{\nu,\psi} C'_{\psi'} \|B\|_{C^\gamma} \sum_{l = -\infty}^\infty \min\{|j+l|^{-\gamma},1\} \min\{|l|^{-\eta},1\}\\
	&\leq C_{\nu,\psi} C'_{\psi'} \|B\|_{C^\gamma} C'' \min\{|j|^{-(\eta + \gamma - 1)},1\}
	\end{align*}
	for some $C''$ depending on $\gamma, \eta$, giving what is required.
\end{proof}

This is all we need to prove Lemma \ref{l:BakerLemma}.
\begin{proof}[Proof of Lemma \ref{l:BakerLemma}]
	By the definition of $\rho_t$ in \eqref{eq:CondMeasDef},
	\[ \int_{D} A(x,y)\,\dd\rho_t(x,y) = \int_0^1 A(\psi(y)-t,y)\,\dd\nu_0(y), \]
	and so
	\begin{align*}	\rho_0(A\circ b^n\, B) - \rho(B) \rho(A) &= 
	\int_0^1 (A\circ b^n)(\psi(y),y)\, \dd \sigma(y) - \int_{D} A\,\dd\rho \int_0^1 \dd \sigma(y)
	\end{align*}
	where $\sigma := B(\cdot,\psi^{-1}(\cdot))\nu_0$. 
	
	Since by assumption, $(\nu_0,\psi)$ has $(\eta, C_{\nu_0,\psi})$ Fourier decay, $\sigma$ has $(\eta + \gamma - 1, C_{\nu_0,\psi} C)$ decay for some $C$ depending on $\eta, \gamma, \psi$ as a result of Lemma~\ref{l:Prop8}. From this, $\sigma$ is also atomless. An application of Proposition~\ref{p:Prop7} gives us our result.
\end{proof}

Finally, the construction of $\rho$ also allows us to prove the corollary:
\begin{proof}[Proof of Corollary~\ref{c:BakerCover}]
	Using Proposition~\ref{p:ExponentialHausdorffDistance} it is enough to prove that the SRB measure $\rho$ is lower-Ahlfors regular. Since $\rho$ is a Cartesian product of Lebesgue measure with the Gibbs measure $\nu_0$, $\rho$ is lower-Ahlfors regular if $\nu_0$ is.
	
	Fix $\delta >0$, and let $M = \lceil \tfrac{\log \delta}{\log \mu} \rceil$. For any $y$ in the support of $\nu_0$ we have that there exists a composition of $M$ contractions in $\{v_j\}$, which we notate as $v_{\mathbf j}$, such that $y = v_{\mathbf j}(y_M)$ for some $y_M \in [0,1]$. Set inclusion tells us that we must therefore have $y \in v_{\mathbf j}([0,1])$.
	
	The uniform contraction of the $v_j$ by a factor of $\nu$ means that the diameter of this set $v_{\mathbf j}([0,1])$ must be bounded by $\mu^M$, which by construction is smaller than $\delta$. Hence, $v_{\mathbf j}([0,1])$ must be contained in the ball $B(y,\delta)$. 
	
	On the other hand, the $\nu_0$-measure of this set $v_{\mathbf j}([0,1])$ can be given using the constitutive relation of $\nu_0$ as $k^{-M} \nu_0([0,1]) = k^{-M}$.
	
	We can therefore say that 
	\begin{align*}
	\nu_0(B(y,\delta)) &\geq \nu_0(v_{\mathbf j}([0,1]))\\
	& = k^{-M} \geq k^{-1} \delta^{-\log k/\log \mu},
	\end{align*}
	as required for lower Ahlfors regularity.
	
	Since the constants here are independent of $t$, we obtain the uniform-in-$t$ convergence asked for.
\end{proof}

\section{Proofs of some integration lemmas}\label{a:baker}

Here we collect some proofs of lemmas used in Appendix~\ref{a:baker} involving integration.
\begin{proof}[Proof of Lemma~\ref{l:DCT}]
	For $l \in \N^+$ the $1$-periodic Fej\'er kernel is given by 
	\[ F_l(x) = \frac{1 - \cos 2\pi l x}{l(1 - \cos 2\pi x)} = \sum_{j=-l}^{l} \left(1 - \frac{|j|}{l}\right) e^{2\pi i j x}. \]
	It is non-negative with total integral equal to $1$, and for all $u \in [0,1/2]$,
	\[ \lim_{l\to\infty}\int_{-u}^u F_l\,\dd x = 1. \]
	As a result, $\phi \ast F_l$ converges pointwise to $\phi$ as $l \to \infty$ at all points of continuity of $\phi$: which is to say, $\nu$-almost everywhere, since $\phi$ has a finite number of jumps and $\nu$ has no atoms. Furthermore, the functions $\phi \ast F_l$ are uniformly bounded by the constant function $\| \phi \|_\infty$ (which is $\nu$-integrable). As a result, we can apply the dominated convergence theorem to say that 
	\begin{equation}
	\int \phi \,\dd\nu = \lim_{l\to\infty} \int_0^1 \phi \ast F_l\,\dd\nu \label{eq:DCT}
	\end{equation}
	Now,
	\begin{align}
	\int_0^1 \phi \ast F_l\,\dd\nu &= \int_0^1 \sum_{j=-l}^l \left(1 - \frac{|j|}{l}\right) \hat \phi_{-j} e^{-2\pi i j x} \,\dd\nu(x)\notag\\
	&= \sum_{j=-l}^l \left(1 - \frac{|j|}{l}\right) \hat\phi_{-j} \hat\nu_{j},\label{eq:CesaroSum}
	\end{align}
	where in the last line we could interchange integration and (finite) summation. Now, \eqref{eq:CesaroSum} is none other than the $l$th C\'esaro sum of $\hat\psi_{-j}\hat\nu_j$, whose limit is therefore the full sum, the full sum being absolutely convergent. Substituting this limit into \eqref{eq:DCT} we obtain \eqref{eq:IntegralFourierDecomposition} as required.
\end{proof}

\begin{proof}[Proof of Lemma~\ref{l:HolderBVFourierDecay}]
	We can divide up the interval of integration into $j$ even pieces as such:
	\[ \int_0^1 \phi(x) e^{-2\pi ijx}\,\dd x = \sum_{l=0}^{|j|-1} \int_{l/|j|}^{(l+1)/|j|} \phi(x) e^{-2\pi ijx}\, \dd x. \]
	We always have on any of these segments that 
	\begin{equation}\left| \int_{l/j}^{(l+1)/j} \phi(x) e^{-2\pi ijx}\, \dd x \right| \leq \frac{1}{|j|} \|\phi\|_\infty. \label{eq:BVContribution}\end{equation}
	Let the index set of segments with jumps be
	\[ J_j = \left\{ l: \left(\tfrac{l}{|j|},\tfrac{l+1}{|j|}\right) \cap X \neq \emptyset\right\}.\]
	Clearly we have that the cardinality of $J_j$ is smaller than that of $X$ for all $j$. For $l \notin J_j$, we can do the usual H\"older continuity bound on the integral, using to begin with that $e^{2\pi ijx}$ is mean zero:
	\[ \int_{l/|j|}^{(l+1)/|j|} \phi(x) e^{-2\pi ijx}\,\dd x = \int_{l/|j|}^{(l+1)/|j|} \left(\phi(x) - \phi(l/|j|)\right) e^{-2\pi ijx} \, \dd x \]
	and then that
	\[ | \phi(x) - \phi(l/|j|) | \leq |\phi|_{C^\alpha(X^c)} (x - l/|j|)^\alpha \leq |\phi|_{\alpha,X^c} |j|^{-\alpha} \]
	to get that
	\begin{equation}
	\left| \int_{l/|j|}^{(l+1)/|j|} \phi(x) e^{-2\pi ijx}\,\dd x \right| \leq |\phi|_{C^\alpha(X^c)} |j|^{-\alpha-1}. \label{eq:HolderContribution}
	\end{equation}
	Combining \eqref{eq:BVContribution} for the segments with jumps and \eqref{eq:HolderContribution} otherwise, we get that 
	\begin{align*}
	\left| \int_0^1 \phi(x) e^{-2\pi ijx}\,\dd x \right| &\leq |J_j| \frac{1}{|j|} \|\phi\|_\infty + (|j| - |J_j|) |\phi|_{C^\alpha(X^c)} |j|^{-\alpha-1} \\
	& \leq |X| \|\phi \|_\infty |j|^{-1} + |\phi|_{C^\alpha(X^c)} |j|^{-\alpha}
	\end{align*}
	as required.
\end{proof}

	\section{Proof of results in Section~\ref{s:Definition}}\label{a:CoveringProofs}
	
	\begin{proof}[Proof of Proposition~\ref{p:ConservativeCM}]
		Since $f$ is conservative and topologically mixing it must have one physical measure $\rho$ which is Lebesgue measure.
		
		Because $H$ has no critical points on $\ell_H$, the conditional measure $\rho(x\mid H(x) = 0)$ can be well-defined as a $C^0$-weak limit of 
		\begin{equation} \mu_\delta(x) = \frac{1}{c(\delta)} \psi(|H(x)|/\delta), \label{eq:MuDelta} \end{equation}
		for some bump function $\psi \in C^\infty_c$, with $c(\delta) = \mathcal{O}(\delta^{d})$.
		
		We know that the stable vector bundle of $f$ is continuous and $\ell_H$ is compact, so its tangent space must be {\it uniformly} transverse to stable vector fields, and so we can find a set of admissible leaves as defined in \cite{Gouezel06} uniformly transverse to $\ell_H$. From the definition conditional measure is the weak limit of the following set of densities 
		
		It can be shown by a computation that the family \eqref{eq:MuDelta} is convergent in the $\mathcal{B}^{1,1}$ norm of \cite{Gouezel06}, including when multiplied by $C^1$ functions $B$, and so the conditional measure lies in $\mathcal{B}^{1,1}$, on which the Perron-Frobenius operator of $f$ has a spectral gap \cite[Theorem~2.3]{Gouezel06}. Hence, exponential conditional mixing obtains with $r=1$.
	\end{proof}
	
		\begin{proof}[Proof of Proposition~\ref{p:HausdorffDistance}]
		We will consider convergence of $d_{\rm Haus}(\Lambda,\overline{T^n(\ell_H \cap \Lambda)})$, as it is the same thing as without the set closure.
		
		From its definition, the conditional measure's support $\supp\mu$ must be contained in $\ell_H \cap \Lambda$, and so $\supp T^n_*\mu$ is contained in $\overline{T^n(\ell_H \cap \Lambda)}$.
		
		Now, $\overline{T^n(\ell_H \cap \Lambda)} \subset \overline{T^n(\Lambda)} \subseteq \overline{\Lambda} = \Lambda$, so it is enough to show that 
		\begin{equation} \lim_{n\to\infty} \inf_{x\in\Lambda} d(x,\supp T^n_*\mu) = 0. \label{eq:HausdorffDistance}\end{equation}
		
		Fix $\epsilon > 0$, and let $\{B(\xi,\epsilon)\}_{\xi \in \Xi}$ be a finite open cover of $\Lambda$. Then, if for each $\xi \in \Xi$ we can show that $B(\xi,\epsilon)$ has positive $T^n_*\mu$-measure for every $n$ large enough, we have that $d_{\rm Haus}(T^n(\ell_H\cap \Lambda),\Lambda) < 2\epsilon$ for these $n$, and so we are done.
		
		Conditional mixing implies that if $\psi_{\xi,\epsilon}$ is any $C^\infty$ non-negative bump function bounded by $1$ whose support is $B(\xi,\epsilon)$, then
		\[ \lim_{n\to\infty} \int \psi_{\xi,\epsilon} \circ T^n \,\dd\mu = \int \psi_{\xi,\epsilon} \dd\rho \int\dd\mu > 0, \]
		because $\psi_{\xi,\epsilon}>0$ on an open set overlapping with $\Lambda = \supp\rho$. This means that for $n$ large enough,
		\begin{equation*}
		0 = \int \psi_{\xi,\epsilon} \circ T^n\,\dd\mu
		= \int \psi_{\xi,\epsilon} \,\dd T^n_*\mu \leq T^n_*\mu(B(\xi,\epsilon)),
		\end{equation*}
		as required.
	\end{proof}
	
	\begin{proof}[Proof of Proposition~\ref{p:ExponentialHausdorffDistance}]
		By translation and dilation we can construct $C^\infty$ bump functions $\psi_{\xi,\epsilon}$ such that $\psi_{\xi,\epsilon} = 1$ on $B(\xi,\epsilon/2)$, and their $C^r$ norms are bounded $\| \psi_{\xi,\epsilon} \|_{C^1} \leq K \epsilon^{-r}$ for constant $K$.
		
		Lower-Ahlfors regularity of $\rho$ gives us that
		\[ \int \psi_{\xi,\epsilon}\,\dd\rho \geq \rho(B(\xi,\epsilon/2)) \geq C (\epsilon/2)^d, \]
		and the exponential conditional mixing assumption then gives us that for all $\xi,\epsilon$,
		\[ \int \psi_{\xi,\epsilon}\circ T^n \,\dd\mu > C (\epsilon/2)^d - C \xi^n K \epsilon^{-r}. \]
		This will all be positive when $\epsilon \geq K_1 (\xi^{1/(d+r)})^n$ for some constant $K_1$, giving us uniform exponential decay in $n$ of the bound $2\epsilon$ on the Hausdorff distance.
	\end{proof}

\section{Proof of Proposition~\ref{p:CriticalOrbitSingularities}}\label{a:CritOrbitSing}

To prove Proposition~\ref{p:CriticalOrbitSingularities} we will first require the following:
\begin{proposition}\label{p:UnstableManifoldExpansion}
	There exist $C > 0$, $\lambda > 1$ such that for all $I \in \hat\L$ and $n \geq 0$ such that $f^m I \cap \S = \emptyset$ for $m < n$,
	\[ |f^n I| > C \lambda^n |I|. \]
\end{proposition}
\begin{proof}
	The segments $I$ are unstable manifolds: because $f$ is piecewise uniformly hyperbolic \cite{Young85}, these segments are eventually expanded by the action of $f$.
\end{proof}

\begin{proof}[Proof of Proposition~\ref{p:CriticalOrbitSingularities}]
	Define the observable on $\vec \Lambda$
	\begin{align}
	P(\vec I_{p,q}, t) &= \cha(\M_{p} \neq \M_{\pi((\vec I_{p,q}, t))})
	\end{align}
	This measures whether $\vec I$ is cut by $\S$ on the left-hand side. 
	Clearly, if $P(\vec f^{-n}(\vec I, t)) = 1$
	, then $p_{\vec I} = f^n(s)$ 
	for some $s \in \S$. This holds {\it a fortiori} if
	\begin{align} \lim_{N\to\infty} \frac{1}{N} \sum_{n=1}^{N} P(\vec f^{-n}(\vec I, t)) &> 0 \label{eq:PObsBirkhoff}
	\end{align}
	In the following we will show that this limit is almost always positive.
	
	Fix $\epsilon > 0$, and let $\vec\Lambda_{\rm erg}$ be an ergodic component of $\vec \Lambda$. Because almost all unstable manifolds have positive measure, the definition of $\vec \Lambda_{\rm erg}$, there exists a set $E \subset \vec \Lambda$ of positive measure such that $|\vec I| > \epsilon$ for all $(\vec I, t) \in E$. It is clear that we can choose $E = E_I \times (0,1)$, $E_I$ being a collection of directed segments.
	
	
	We know that for any segment $\vec J$, if $\vec J$ does not cross $\S$ then $\vec f(\vec J,t) = (f(\vec J),t)$. By Proposition~\ref{p:UnstableManifoldExpansion}, this means that for $\vec I \in A_I$, $f^n(\vec I, t) = (f^n(\vec I) t)$ with
	$ |f^n(\vec I)| \geq C \epsilon \lambda^n$, 
	unless $\vec I$ is cut by $\S$ for some $m < n$. However, there is some $m_*$ sufficiently large such that $C \epsilon \lambda^{m_*}$ is greater than the diameter of the attractor $\Lambda$, which brings about a contradiction, as $f^n(I) = \Wul(f^n(\pi(\vec I,0.5)))$ is an segment subset of $\Lambda$. Thus, $f^m(\vec I)$ is cut by $\S$ for some $m > n$. 
	
	Now, recalling that our segments are open segments, $\S$ will cut $f^m \vec I$ at some point $t' \in (0,1)$. As a result, $P(\vec f^{m+1}(I,t)) = 1$ for $t \in (0,t')$
	. This holds for any $\vec I \in E_I$. 
	
	The consequence is that, for at least one $m'$ between $0$ and $m_*$, there exists $B_P
	\subset f^{m'+1}(A) \subset \vec \Lambda_{\rm erg}$ of positive measure such that $P = 1$ on $B_P$. By applying the Birkhoff ergodic theorem to $(\vec f^{-1},\vec \Lambda_{\rm erg},\vec \rho|_{\vec \Lambda_{\rm erg}})$, the limit (\ref{eq:PObsBirkhoff}) must hold for $\vec \rho$-almost all points on on $\vec \Lambda_{\rm erg}$.
	
	By taking a union over all ergodic components of $\vec \Lambda$ we therefore have that for almost all $\vec I_{p,q} \in \vec \Lambda$, $p_{\vec I} = f^{n}(s)$ for some $s \in \S$. This is to say that $p_{\vec I}$ lies in the forward orbit of the singular line. Using that reversing the direction of segments is a measure isometry, this result equally holds true for $q$.
\end{proof}
%
%


\subsection*{Acknowledgements}

This research has been supported by the European Research Council (ERC) under the European Union's Horizon 2020 research and innovation programme (grant agreement No 787304), as well as by the ETH Zurich Institute for Theoretical Studies.

The author thanks Viviane Baladi for her comments on early stages of the manuscript.

\subsection*{Data availability statement}
The datasets generated and/or analysed during the current study are available from the author on reasonable request.

\subsection*{Conflict of interest statement}
The author has no conflicts of interest to declare.

\bibliographystyle{plain}
\bibliography{../lozi-paper/lozi}

\begin{thebibliography}{10}

\bibitem{Baladi92}
Viviane Baladi.
\newblock Optimality of {R}uelle's bound for the domain of meromorphy of
  generalized zeta functions.
\newblock {\em Portugaliae mathematica}, 49(1):69--83, 1992.

\bibitem{Baladi00}
Viviane Baladi.
\newblock {\em Positive transfer operators and decay of correlations},
  volume~16.
\newblock World scientific, 2000.

\bibitem{BaladiGouezel09}
Viviane Baladi and S{\'e}bastien Gou{\"e}zel.
\newblock Good {B}anach spaces for piecewise hyperbolic maps via interpolation.
\newblock In {\em Annales de l'Institut Henri Poincar{\'e} C, Analyse non
  lin{\'e}aire}, volume~26, pages 1453--1481. Elsevier, 2009.

\bibitem{Baranski20}
Krzysztof Bara{\'n}ski, Yonatan Gutman, and Adam {\'S}piewak.
\newblock A probabilistic {T}akens theorem.
\newblock {\em Nonlinearity}, 33(9):4940, 2020.

\bibitem{Bonatti16}
Christian Bonatti, Sylvain Crovisier, Lorenzo~J D{\'\i}az, and Amie Wilkinson.
\newblock What is\ldots a blender?
\newblock {\em Notices of the American Mathematical Society},
  63(10):1175--1178, 2016.

\bibitem{Bourgain17}
Jean Bourgain and Semyon Dyatlov.
\newblock Fourier dimension and spectral gaps for hyperbolic surfaces.
\newblock {\em Geometric and Functional Analysis}, 27(4):744--771, 2017.

\bibitem{Chernov95}
Nikolai~I Chernov.
\newblock Limit theorems and {M}arkov approximations for chaotic dynamical
  systems.
\newblock {\em Probability Theory and Related Fields}, 101(3):321--362, 1995.

\bibitem{DemersLiverani08}
Mark Demers and Carlangelo Liverani.
\newblock Stability of statistical properties in two-dimensional piecewise
  hyperbolic maps.
\newblock {\em Transactions of the American Mathematical Society},
  360(9):4777--4814, 2008.

\bibitem{Doucet01}
Arnaud Doucet, Nando De~Freitas, Neil~James Gordon, et~al.
\newblock {\em Sequential Monte Carlo methods in practice}, volume~1.
\newblock Springer, 2001.

\bibitem{Gottwald16}
Georg~A Gottwald, Caroline~L Wormell, and Jeroen Wouters.
\newblock On spurious detection of linear response and misuse of the
  fluctuation--dissipation theorem in finite time series.
\newblock {\em Physica D: Nonlinear Phenomena}, 331:89--101, 2016.

\bibitem{Gouezel06}
S{\'e}bastien Gou{\"e}zel and Carlangelo Liverani.
\newblock Banach spaces adapted to anosov systems.
\newblock {\em Ergodic Theory and Dynamical Systems}, 26(1):189--217, 2006.

\bibitem{Hochman15}
Michael Hochman and Pablo Shmerkin.
\newblock Equidistribution from fractal measures.
\newblock {\em Inventiones mathematicae}, 202(1):427--479, 2015.

\bibitem{Misiurewicz80}
Michal Misiurewicz.
\newblock Strange attractors for the {L}ozi mappings.
\newblock {\em Annals of the New York Academy of Sciences}, 357(1):348--358,
  1980.

\bibitem{Mosquera22}
Carolina~A Mosquera and Andrea Olivo.
\newblock Fourier decay of self-similar measures on the complex plane.
\newblock {\em arXiv preprint arXiv:2207.11570}, 2022.

\bibitem{Mosquera18}
Carolina~A Mosquera and Pablo Shmerkin.
\newblock Self-similar measures: asymptotic bounds for the dimension and
  {F}ourier decay of smooth images.
\newblock {\em Annales Academi\ae\ Scientiarum Fennic\ae\ Mathematica},
  43:823--834, 2018.

\bibitem{Ruelle18}
David Ruelle.
\newblock {Linear response theory for diffeomorphisms with tangencies of stable
  and unstable manifolds---a contribution to the {G}allavotti-{C}ohen chaotic
  hypothesis}.
\newblock {\em Nonlinearity}, 31(12):5683, 2018.

\bibitem{Sahlsten20}
Tuomas Sahlsten and Connor Stevens.
\newblock Fourier transform and expanding maps on {C}antor sets.
\newblock {\em arXiv preprint arXiv:2009.01703}, 2020.

\bibitem{Schmeling98}
J{\"o}rg Schmeling and Serge Troubetzkoy.
\newblock Dimension and invertibility of hyperbolic endomorphisms with
  singularities.
\newblock {\em Ergodic Theory and Dynamical Systems}, 18(5):1257--1282, 1998.

\bibitem{lozi}
Caroline~L Wormell.
\newblock On convergence of linear response formulae in some piecewise
  hyperbolic maps.
\newblock {\em {\it arXiv} preprint arXiv:2206.09292}, 2022.

\bibitem{Young85}
Lai-Sang Young.
\newblock {B}owen-{R}uelle measures for certain piecewise hyperbolic maps.
\newblock In {\em The Theory of Chaotic Attractors}, pages 265--272. Springer,
  1985.

\end{thebibliography}

\end{document}